\documentclass{article}
\usepackage{amsmath, amssymb, mathrsfs, amsthm}
\usepackage{tikz}
\usetikzlibrary{calc}
\usepackage{enumerate}

      \theoremstyle{plain}
      \newtheorem{theorem}{Theorem}[section]
      \newtheorem*{theorem*}{Theorem}
	  
	  \newtheorem*{claim*}{Claim}
      \newtheorem{proposition}[theorem]{Proposition}
      \newtheorem*{proposition*}{Proposition}
      \newtheorem{lemma}[theorem]{Lemma}
      \newtheorem*{lemma*}{Lemma}
      \newtheorem{corollary}[theorem]{Corollary}
      \newtheorem*{corollary*}{Corollary}
      
      \theoremstyle{definition}
      \newtheorem{definition}[theorem]{Definition}
      \newtheorem*{definition*}{Definition}
      
      \theoremstyle{remark}
      \newtheorem{remark}[theorem]{Remark}
      \newtheorem*{remark*}{Remark}
      
      \newtheorem*{example*}{Example}

\newcommand{\R}{\mathbb{R}}
\newcommand{\C}{\mathbb{C}}
\newcommand{\N}{\mathbb{N}}

\newcommand{\M}{\mathscr{M}}

\newcommand{\BSF}{{\widehat{B^s_{p,q}}}}

\renewcommand{\d}{\partial}
\renewcommand{\tilde}{\widetilde}
\newcommand{\abs}[1]{\left| #1 \right|}
\newcommand{\norm}[1]{\left\lVert #1 \right\rVert}
\newcommand{\supp}{\operatorname{supp}}

\renewcommand{\O}{\Omega}
\renewcommand{\hat}{\widehat}
\newcommand{\restr}[1]{|_#1}

\newcommand{\nn}{\nonumber}
\newcommand{\na}[1]{\noalign{\noindent #1}}
\newcommand{\espace}{\vspace*{2mm}}
\newcommand{\Sum}{\ensuremath{\mathop{\sum}\limits}}
\newcommand{\Sup}{\ensuremath{\mathop{\sup}\limits}}
\newcommand{\Int}{\ensuremath{\mathop{\int}\limits}}
\newcommand{\h}[1]{\hat{#1}}

\title{Corners always scatter}

\author{Eemeli Bl{\aa}sten\thanks{Research supported by the Finnish Centre of Excellence in Inverse Problems Research}\\ University of Helsinki \and Lassi P{\"a}iv{\"a}rinta\thanks{Research supported by the ERC 2010 Advanced Grant 267700 and the Finnish Centre of Excellence in Inverse Problems Research}\\ University of Helsinki \and John Sylvester\thanks{Research supported in part by the NSF grant DMS-1007447}\\ University of Washington}

\date{June 2013}

\begin{document}
\maketitle

\section{Abstract}
We study time harmonic scattering for the Helmholtz
equation in $\R^n$. We show that certain penetrable
scatterers with rectangular corners scatter every incident
wave nontrivially. Even though these scatterers have
interior transmission eigenvalues, the relative scattering
(a.k.a. far field) operator has a trivial kernel and
cokernel at every real wavenumber.

\section{Introduction}
The diffraction of light around corners and edges, and
through slits, provided the first evidence for the wave
nature of light. The diffraction patterns caused by plane
waves incident on corners or edges and were among the first
scattered waves to be calculated \cite{SommerfeldOptics}.
Geometric optics expansions for scattered waves
\cite{KellerLewis} reveal the presence of scattered waves
in regions where the simple theory of optics does not.
Much of our understanding of classical electromagnetism
is based on these patterns. This is why a stealth airplane
is built to minimize the scattering from corners and
edges.

Although the single frequency inverse scattering problem
has a unique solution, the wave scattered from a single
incident wave does not contain enough information to
determine an obstacle or a penetrable
scatterer. In many cases, the same scattered wave might have
been scattered by a scatterer supported on a smaller
set. In this paper we will show that, a penetrable scatterer
whose support contains a right angle corner as an extreme
point of its convex hull will scatter any incident wave
nontrivially. This result has the following consequence. Suppose one scatterer contains such a right angle corner and a second scatterer does not contain the corner point in its convex hull. Then the two scatterers have no scattered wave in common. The ranges of their scattering operators are disjoint.\\

The same is not true for a compactly supported obstacle. A
square with sidelength \(\pi\) has Dirichlet
eigenfunctions
\begin{eqnarray}
\nn
  v(x,y) &=& 4\sin(nx)\sin(my) 
  \\\nn
         &=& e^{i(nx+my)}+e^{-i(nx+my)}-e^{i(nx-my)}-e^{-i(nx-my)}
\end{eqnarray}
which means that a sum of four plane waves incident on
this sound soft obstacle produces no scattered wave. Even
though the obstacle has corners, it is invisible to this
incident pattern.

For a penetrable scatterer, the \textit{interior
  transmission eigenvalues} play the same role that the
Dirichlet eigenvalues play for the sound soft
obstacle. Any compactly supported \(L^{\infty}\) scatterer
with positive contrast has infinitely many
\textit{interior transmission eigenvalues}. This implies
the existence of wavenumbers \(k\) for which there exist
\(L^{2}\) incident waves defined on the support of the
scatterer, which produce no scattered wave.

In the spherically symmetric case, the existence of
such  wavenumbers has been known for a long time
\cite{coltonMonk88,coltonPaivarintaSylvester07}. In this
case, the corresponding incident waves extend to
\(\R^{n}\) as Herglotz wavefunctions, so the classical
relative scattering operator has a nontrivial kernel. This is
significant because many reconstruction algorithms in inverse
scattering theory, such as the linear sampling method of
Colton and Kirsch \cite{coltonKirsch96}, and the
factorization method of Kirsch \cite{kirschGrinberg08},
 will work correctly only if the kernel and
cokernel of the relative scattering operator is trivial.

The existence of finitely many interior transmission
eigenvalues for general (non-spherically symmetric)
scatterers with positive contrast was first shown in
\cite{paivarintaSylvester08} in 2008, extended to
infinitely many in \cite{CakoniGintidesHaddar} in 2010,
and generalized to higher order operators in \cite{HKOP1}.
If the support contains a right angle corner, we prove that
these incident waves cannot extend to any open
neighborhood of the corner.  The interpretation is that
these incident waves could only be produced by sources
located on the boundary of the scatterer, but not by any
combination of sources located outside an open
neighborhood of the scatterer. One particular corollary
is that the linear
sampling and factorizations methods, which utilize only
Herglotz wavefunctions as incident waves, will work
successfully for such scatterers.

Our analysis relies on two new theorems that are of
independent interest. We give a new construction of the
so-called complex geometric optics solutions for the
Helmholtz equation, combining the techniques of
Agmon-Hormander \cite{AgmonHormander} and Ruiz
\cite{RuizNotes} to work in \(L^{p}\) based Besov
spaces. This allows us to improve the local regularity of
these solutions without sacrificing the decay as a
function of complex frequency.

The second theorem states the the Laplace transform of a
harmonic polynomial cannot vanish identically on its
complex characteristic variety \(\{\zeta\;\big|
\zeta\cdot\zeta = 0\}\). This is a generalization of the
well-known fact that the Fourier Transform of the solution
to a homogeneous constant coefficient partial differential
equation is supported on the real characteristic variety
of the differential operator, so that it cannot vanish on
that set unless it is identically zero. Although the
support statement cannot be true for the Laplace transform
because it is an analytic function, we show, in the special
case of the Laplacian, that only the zero harmonic polynomial can
vanish identically on this variety. A proof of this theorem for a
general second order elliptic operator with constant coefficients 
would remove the restriction of our results to right angle corners.\\

The classical scattering of time harmonic waves by a
penetrable medium can be modeled by the Helmholtz equation
\begin{equation}\nn
(\Delta + k^2 n^2)u = 0 \quad \text{in } \R^n,
\end{equation}
where $n(x)$ denotes the index of refraction. In this
model, we seek the total wave as 
\begin{equation}\nn
u = v^0 + u^+
\end{equation}
where $v^0$ is the \emph{incident wave} and $u^+$ the
outgoing \emph{scattered wave}. This means that
\begin{equation}
\label{FreeHelm}
(\Delta + k^2)v^0 = 0 \quad \text{in } \R^n
\end{equation}
and therefore that
\begin{equation}
\label{Helm}
(\Delta + k^2) u^+ = k^2 m(v^0 + u^+)
\end{equation}
We assume that the \emph{contrast} $m$, defined by
\begin{equation}\nn
n^2 = 1 - m,
\end{equation}
is compactly supported. The relative scattering operator
maps the asymptotics of Herglotz incident waves to the asymptotics
of scattered waves. A Herglotz incident wave is defined to
be a solution to (\ref{FreeHelm}) of the form
\begin{equation}\nn
v^0(x) = \int_{S^{n-1}} g_0(\theta) e^{ik\theta\cdot x} d\sigma(\theta),
\end{equation}
for some $g_0 \in L^2(S^{n-1})$. The Herglotz incident
waves can be characterized as the solutions to
(\ref{FreeHelm}) whose Fourier transforms belong to the
Besov space $B^{-1/2}_{2,\infty}(\R^{n})$
\cite{AgmonHormander}\footnote{We will give these
  defnitions and make use of these norms in section 
  \ref{sectionFunctionSpaces}. See also\cite{Stein,
    Triebel1}}.  These incident waves have well-defined
asymptotics at infinity
  
\begin{equation}\nn
v^0(r\theta) \sim \frac{e^{ikr}}{(ikr)^\frac{n-1}{2}} g_0(\theta) + \frac{e^{-ikr}}{(-ikr)^\frac{n-1}{2}} g_0(-\theta).
\end{equation}
The scattered wave $u^+$ also has asymptotics at infinity
\begin{equation}\nn
u^+(r\theta) \sim \frac{e^{ikr}}{(ikr)^\frac{n-1}{2}} \alpha^+(\theta)
\end{equation}
and the relative scattering operator $S(k)$ maps
\begin{equation}\nn
S(k): L^2(S^{n-1}) \ni g_0 \mapsto \alpha^+ \in L^2(S^{n-1}).
\end{equation}
For each \(k\), the operator \(S(k)\) is compact and
normal; it never has a bounded inverse, but a number of
methods in inverse scattering succeed only if the kernel,
and hence cokernel, of $S(k)$ are trivial. If the contrast
$m(x)$ in \eqref{Helm} is compactly supported, then a
nontrivial kernel implies that $k^2$ is an \emph{interior
  transmission eigenvalue} (ITE) for any domain $\O$ that
contains the support of $m$ in its interior. This means
that there are nontrivial $u^+$ and $v^0$ satisfying

\begin{eqnarray}\label{eq:2}
(\Delta + k^2) v^0 &=& 0 \quad \text{in } \O
\\\label{eq:3}
(\Delta + k^2(1-m)) u^+ &=& k^2mv^0 \quad \text{in } \O
\\\label{eq:4}
u^+|_{\d\O} = 0, && \,\,\tfrac{\d u^+}{\d \nu}|_{\d\O} = 0
\end{eqnarray}

In general, the eigenfunctions \(u^{+}\) belong to \(
H^{2}_{0}(\Omega)\), and therefore extend to all of
\(\R^{n}\) as a function which is zero outside
\(\Omega\). The waves \(v^{0}\), in general, are only known to
satisfy \(v^{0}\in L^{2}(\Omega)\) and \(\Delta v^{0}\in
L^{2}(\Omega)\). We will refer to  \(v_{0}\) as an
\emph{interior incident wave}, to emphasize that it is
only defined in  \(\Omega\).  In particular, the ITE's depend on both
\(m\) and \(\Omega\).  If $m(x)\restr{\O} > 0$, there
exist infinitely many real ITE's. However, for the same
scatterer \(m\), and a slightly larger domain
\(\Omega_{1}\), there may exist no real ITE's, because the
interior incident waves may not extend to \(\Omega_{1}\).\\

In the spherically symmetric case ($m = m(\abs{x})$),
every interior incident wave
$v^0$ extends to $\R^n$ as a spherical harmonic times a
Bessel function, which is a Herglotz wavefunction, so that
the relative scattering operator genuinely has a
nontrivial kernel and cokernel. We will say that \(k\) is a
\emph{non-scattering wavenumber} whenever the relative
scattering operator \(S(k)\) has a nontrivial
kernel. Although all scatterers with positive contrasts
have infinitely many real ITE's (with \(\Omega\) equal to
the support of the contrast), no non-spherically symmetric
scatterers are known to have \emph{non-scattering
  wavenumbers}.



\smallskip In this paper, we show that, if the contrast
$m(x)$ is the characteristic function of an
$n$-dimensional rectangle times a smooth function which is
nonzero at at least one corner of the rectangle, any
non-scattering interior incident wave $v^0$ does not
extend, as a solution to \eqref{FreeHelm}, to any open
neighborhood of the corner. In particular, no such
scatterer can have non-scattering wavenumbers.

\section{All Corners Scatter}
The theorem below applies to scatterers whose support
contains a corner (a standard right angle corner) as an
extreme point of its convex hull (i.e. there exists a
hyperplane which touches the support of the scatterer at
precisely that corner). We describe this condition in item
\ref{item:3}) below by stating that  \(m\) is the product of a
smooth function and the characteristic function of a
rectangle.  \goodbreak
\begin{theorem}
\label{main}
Suppose that $k \neq 0$, that $K$ is an $n$-dimensional rectangle, and
\begin{enumerate}[i)]
\label{mainThmConditions}
\item\label{item:3} $m = \chi_K \varphi(x)$ with $\varphi \in
  C^\infty(\R^n)$ and $\varphi(x_0) \neq 0$ where
  \(x_{0}\) is a corner of $K$
\item the pair \((u^{+},v^{0})\) are interior transmission
  eigenfunctions of \(m\) in \(\Omega = \supp m\) , i.e
  solutions to (\ref{eq:2}-\ref{eq:3}-\ref{eq:4})
\end{enumerate}
then \(v^{0}\) cannot be extended as an incident wave(i.e
a solution to (\ref{FreeHelm})) to any open  neighborhood
of the corner.
\end{theorem}

\begin{corollary}
A scatterer $m$ which satisfies item \ref{item:3}) has no
non-scattering wavenumbers.
\end{corollary}
\begin{proof}
  If the kernel of $S(k)$ is nontrivial, then there is a
  Herglotz wavefunction \(v^0\) satisfying \eqref{FreeHelm} in
  $\R^n$, and an outgoing $u^+$ satisfying (\ref{Helm}) in
  $\R^n$ with vanishing far field $\alpha^+$. Rellich's
  lemma and unique continuation \cite{RellichLemma}
  guarantee that $u^+$ vanishes outside the support of
  $m$. It follows from the fact that $m\in L^\infty$ and
  $v^0 \in L^2$ that $u^+ \in H^2_{loc}(\R^n)$, and
  therefore the restriction of \(u^{+}\) and its first
  derivative to \(\d\Omega\) must vanish. Hence the pair
  \((u^{+},v^{0})\) are interior transmission
  eigenfunctions in  \(\Omega\), but  \(v^{0}\) extends
  past the corner, contradicting Theorem~\ref{main}.
\end{proof}

We summarize our proof of Theorem \ref{main} in the
following paragraph.
We
 will make use of some
\emph{complex geometric optics} solutions to the
homogeneous version of (\ref{eq:3}). Specifically, if we
multiply equation (\ref{eq:3}) by any solution  \(w\) to 
\begin{equation}
\label{mFree}
\big( \Delta + k^2(1-m) \big) w = 0 
\end{equation}
and integrate by parts, using the fact that  \(u^{+}\) and
its first derivatives vanish on  \(\d \Omega\), we see that 
\begin{equation}
\label{orthoFree}
\int_K w k^2 m v^0 = 0
\end{equation}
Theorem \ref{cgolp} below shows that we may choose \(w\) to be
exponentially decaying as we move into  \(\Omega\) from
the corner, so that the main contribution to the integral
occurs at the corner. If  \(v^{0}\) could be extended to
a neighborhood of the corner, its Taylor series would
necessarily begin with a harmonic polynomial
(Lemma~\ref{TaylorFree}), and the dominant term in the
integral would come from the decaying exponential times that
harmonic polynomial. This would then imply that the
Laplace transform of this harmonic polynomial vanished on
the complex characteristic variety associate to the
Laplacian, and we devote Section \ref{sectionPolynomial}
to the proof of Theorem \ref{LTharmonic}, which says that this
cannot be so.\\

The complex geometric optics solutions we use go back to
\cite{sylvesteruhlmann87}. There have been many
improvements since then, but none provide enough local
regularity to show that their contributions to the
integral in (\ref{orthoFree}) are dominated by the Laplace
transform of the harmonic polynomial. Therefore, we give a
new construction in Section \ref{sec:besovestimates}, combining the
 \(L^{p}\) techniques in  \cite{RuizNotes} with the
 geometric  \(L^{2}\) based constructions in
 \cite{AgmonHormander} to prove
\begin{theorem}
\label{cgolp}
Suppose that $m(x)$ satisfies i) in Theorem
\ref{mainThmConditions}. For any bounded domain \(D\),
and any $2 \leq p < \infty$, there exist constants
\(C\) and  \(r\) such that if  $\rho \in \C^n$ and satisfies
$\rho\cdot\rho = 0$ and  \(|\rho|>r\), there exists $w$ satisfying
\eqref{mFree} in \(D\) of the form
\begin{equation}\label{eq:1}
w = e^{-x\cdot \rho}(1 + \psi)
\end{equation}
with
\begin{equation}
\label{cgolpest}
\norm{\psi}_{L^p(D)} \leq \frac{C}{\abs{\rho}}
\end{equation}
\end{theorem}

It is the statement \(2 \leq p < \infty\) that differentiates
Theorem \ref{cgolp} from previous
constructions. We will need to choose  \(p>n\),
while maintaining the first power of  \(|\rho|\) in the denominator 
for our proof to succeed.\\

The simple lemma below notes that the first term in the
Taylor series of an incident wave at an interior point is
a harmonic polynomial. 
\begin{lemma}
\label{TaylorFree}
Suppose that $v^0 \not\equiv 0$ and  $x_0$ is in an open set where $(\Delta + k^2)v^0 = 0$. Then the lowest order homogeneous polynomial in the Taylor series for $v^0$ at $x_0$ is harmonic.
\end{lemma}
\begin{proof}
The function $v^0$ is real analytic at $x_0$, so its
Taylor expansion doesn't vanish.
We call the lowest order polynomial  \(P_{N}\) and \(v^{N+1}\) is the remainder.
\begin{equation}\nn
\begin{split}
v^0(x) &= P^N(x-x_0) + v^{N+1}(x)  \\
\Delta v^0(x) &= \Delta P^N(x-x_0) + \Delta v^{N+1}(x)  \\
&= Q^{N-2}(x-x_0) + q^{N-1}(x) 
\end{split}
\end{equation}
where $P^N$ and $Q^{N-2}$ are homogeneous polynomials of degree $N$ and $N-2$ respectively, and
\begin{equation}\nn
\begin{split}
\abs{v^{N+1}(x)} &\leq c \abs{x-x_0}^{N+1} \\
\abs{q^{N-1}(x)} &\leq c \abs{x-x_0}^{N-1}.
\end{split}
\end{equation}
We may assume that $N\geq 2$ as all polynomials of degree
less than two are harmonic. In this case, it follows from
\begin{equation}\nn
\Delta v^0 = - k^2 v_0
\end{equation}
that
\begin{equation}\nn
\abs{Q^{N-2}(x-x_0)} = \abs{- q^{N-1}(x) - k^2(P^N + v^{N+1})} \leq c \abs{x-x_0}^{N-1},
\end{equation}
but $Q^{N-2}$ is homogeneous of order $N-2$, so must be zero.
\end{proof}

The final main ingredient, which we will prove in Section
\ref{sectionPolynomial}, concerns the Laplace transform of
a homogeneous harmonic polynomial, i.e.

\begin{equation}\nn
\widehat{P}(\rho) := \int_{x>0} e^{-x\cdot \rho} P^N(x) dx
\end{equation}
where the notation  \(x>0\) means that every component of
\(x\) is greater than  \(0\). We also use the notation
\(\frac{1}{\rho}\) to denote the vector in  \(\C^{n}\)
whose components are the reciprocals of the components of  \(\rho\).

\begin{theorem}
\label{LTharmonic}
The Laplace transform of a nonzero degree \(N\) homogeneous
harmonic polynomial on \(\R^{n}\) is a degree \(N+n\)
homogeneous polynomial \(Q^{N+n}(\frac{1}{\rho})\) of the
reciprocals of the transform variables. If \(n\ge 3\), it
cannot vanish identically on any open subset of the
variety $\rho \cdot \rho = 0$.  If \(n=2\), it cannot
vanish identically on both an open subset of
\(\rho_{1}=i\rho_{2}\) and an open subset of 
\(\rho_{1}= -i\rho_{2}\).
\end{theorem}

Theorem \ref{main} is now a fairly direct consequence.

\begin{proof}[Proof of Theorem \ref{main}]

  Without loss of generality, we will assume that the
  rectangle is located in the positive orthant $\{ x_j >
  0\}$, that $x=0$ is the corner at which $m$ doesn't
  vanish, and that $m(0) = 1$. We choose \(\rho\in\C^{n}\)
  satisfying \(\rho\cdot\rho = 0\) and such that the real
  part of each component \(\rho_{j}> \frac{1}{2\sqrt{n}}\). 
  This guarantees that for each
  \(x\) in the positive orthant
  \begin{equation}
    \nn
    -\Re{x\cdot\rho}<-\tau|x||\rho|
  \end{equation}
  with \(\tau = \frac{1}{2\sqrt{n}}\). Note that the set
  of \(\rho\) that satisfy this condition is an open
  subset of the variety \(\rho\cdot\rho = 0\). Hence, for
  \(n\ge3\), the Laplace transform of any harmonic
  polynomial does not vanish at at least one such
  \(\rho\). For \(n=2\), we note that, an open subset of
  \(\rho\cdot\rho = 0\) contains an open subset of either
  \(\rho_{1}=i\rho_{2}\) or an open subset of \(\rho_{1}=
  -i\rho_{2}\). Our harmonic polynomial cannot vanish on
  both. If it vanishes on one of these, we change \(\rho\)
  to its complex conjugate \(\overline{\rho}\), which is in
  the other, and has the same real part. \\

We insert
the $w$ from Theorem \ref{cgolp}, with this \(\rho\) into \eqref{orthoFree},
obtaining
\begin{eqnarray}\label{eq:5}
0 = \int_\Omega e^{-x\cdot \rho} (1+\psi) m v_{0}.
\end{eqnarray}
Outside a disk of radius  \(\epsilon\)  of the corner, the contribution is
exponentially small
\begin{eqnarray*}
\abs{\int_{\Omega\setminus N_{\epsilon}} e^{-x\cdot \rho} (1+\psi)
m v^0}
\le e^{-\tau\epsilon|\rho|} ||1+\psi||_{2}||mv^{0}||_{2}\le  C e^{-\tau\epsilon|\rho|} 
\end{eqnarray*}
if we choose \(\tau\) large enough. 
Inside the  \(\epsilon\) neighborhood, 
we expand $v^0$ as in Lemma \ref{TaylorFree}, obtaining
\begin{eqnarray*}
 \int_{N_{\epsilon}}e^{-x\cdot \rho} (1+\psi) m\big(P^N(x)
 + v^{N+1}(x)\big)
\end{eqnarray*}

We now rewrite (\ref{eq:5}) as 
\begin{multline}
\label{orthoFree2}
\int_{N_{\epsilon}} e^{-x\cdot \rho} m P^N =
\int_{N_{\epsilon}} e^{-x\cdot\rho} m Q^{N+1}
\tilde{v}^{N+1} + \int_{N_{\epsilon}} e^{-x\cdot\rho} \psi m (P^N
+ Q^{N+1} \tilde{v}^{N+1})\\
 + \int_{\Omega\setminus N_{\epsilon}} e^{-x\cdot \rho} (1+\psi)
m v^0
\end{multline}
where we have rewritten $v^{N+1} = -Q^{N+1}
\tilde{v}^{N+1}$ as a homogeneous polynomial times an
analytic function $\tilde{v}^{N+1}$. Note that
$\tilde{v}^{N+1}$ remains bounded in  \(N_{\epsilon}\)
because  \(v_{0}\)  is analytic in a full neighborhood of
the corner point. The following lemma tells us how the
first two terms on right hand side
of \eqref{orthoFree2} decay as $\abs{\rho} \to \infty$.

\begin{lemma}
  Let $R^N(x)$ be a homogeneous polynomial of degree $N$
  and $\Re \rho_j > 0$ for all $j = 1,\ldots , n$. Then,
  for any  \(f\in L^{p}\) 
\begin{equation}\nn
  \abs{\int_{x > 0} e^{-x\cdot \rho} R^N(x) f(x) dx} \leq
  C \abs{\rho}^{-(N+n) + n/p}\norm{f}_{L^p}
\end{equation}
\end{lemma}
\begin{proof}
Let $\rho = s \cdot \theta$, where $\theta\in\C^n$, $\abs{\theta} = 1$ and $s>0$. Then
\begin{eqnarray}\nn
  \int_{x>0} e^{-s x \cdot \theta} R^N(x) f(x) dx
&=&
  \frac{1}{s^{N+n}} \int_{y>0}
  e^{-y\cdot \theta} R^N(y) f\left(\tfrac{y}{s}\right) dy
\\\nn
  &\leq& \frac{1}{s^{N+n}} \norm{e^{-y\cdot \theta} R^N(y)}_{L^q}
   \norm{f\left(\tfrac{y}{s} \right)}_{L^p}
\\\nn
&=&
  \frac{C_{\theta,n,q,R}}{s^{N+n}} s^{n/p} \norm{f}_{L^p}
\end{eqnarray}
\end{proof}

 The lemma gives us a bound on the first two terms on the
 right hand side of (\ref{orthoFree2})
\begin{multline}\label{eq:6}
   \abs{\int_{N_{\epsilon}}
    e^{-x\cdot \rho}m Q^{N+1}\tilde{v}^{N+1}}
  + \abs{\int_{N_{\epsilon}} e^{-x\cdot\rho} \psi m (P^N + Q^{N+1}\tilde{v}^{N+1})} \\
  \leq \frac{C}{\abs{\rho}^{N+n+1}} \norm{m
    \tilde{v}^{N+1}}_{L^\infty} + \frac{C}{\abs{\rho}^{N+n
      - n/p}} \norm{\psi m \tilde{v}^{N+1}}_{L^p},
\end{multline}
which combines with \eqref{cgolpest} to yield
\begin{equation}\nn
  \leq \norm{m \tilde{v}^{N+1}}_{L^\infty} \left(
    \frac{C}{\abs{\rho}^{N+n+1}} +
    \frac{C}{\abs{\rho}^{N+n-n/p}} \cdot
    \frac{C}{\abs{\rho}} \right)
 \leq
  \frac{C}{\abs{\rho}^{N+n + (1-n/p)}}
\end{equation}
Theorem \ref{cgolp} allows us to choose any $2 \leq
p<\infty$, say $p=2n$, so the right hand side of
(\ref{eq:6}) is bounded by\footnote{This is where we make
  essential use of the  \(L^{p}\)  estimates with  \(p>2\)  for
  \(\psi\) in Theorem \ref{cgolp}. We need
  \(1-\frac{n}{p}\) to be positive in order to show that
  these terms are dominated by the Laplace transform of
  the harmonic polynomial, which is bounded from below by
  \(|\rho|^{-(N+n)}\).}
\begin{equation}\nn
\leq \frac{C}{\abs{\rho}^{N+n+1/2}}.
\end{equation}
and consequently
\begin{equation}
  \label{eq:7}
  \abs{\int_{N_{\epsilon}} e^{-x\cdot \rho} m P^N } \le
  \frac{C}{\abs{\rho}^{N+n+1/2}} +  C e^{-\tau |\rho|\epsilon}
\end{equation}
Because $m(x) - 1$ vanishes at $x=0$, we have
\begin{eqnarray}\nn
\abs{\int_{x>0} e^{-x\cdot \rho} P^N(x) \big( m(x) - 1
  \big) dx}
&=& \abs{\int_{x>0} e^{-x\cdot\rho} \tilde{Q}^{N+1}(x)
  \tilde{m}(x) dx}
\\\nn
&\leq& \frac{C}{\abs{\rho}^{N+n+1}} \norm{\tilde{m}}_{L^\infty}
\end{eqnarray}
combining with equation (\ref{eq:7}) implies
\begin{equation}
  \nn
  \abs{\int_{N_{\epsilon}} e^{-x\cdot \rho}  P^N} \le
  \frac{C}{\abs{\rho}^{n+N+1/2}} +  C e^{-\tau |\rho|\epsilon}
\end{equation}

On the other hand, Theorem \ref{LTharmonic} tells us that
\begin{equation}\nn
\abs{\int_{x>0} e^{-x\cdot \rho} P^N(x) dx } \geq \frac{C}{\abs{\rho}^{N+n}}
\end{equation}
with $C$ nonzero after a suitable choice of $\theta = \rho/\abs{\rho}$, and consequently that 
\begin{equation}\nn
\abs{\int_{N_{\epsilon}} e^{-x\cdot \rho} P^N(x) dx } \geq
\frac{C}{\abs{\rho}^{N+n}} - C e^{-\tau |\rho|\epsilon}
\end{equation}

Hence we arrive at the contradiction that
\begin{equation}\nn
\frac{C}{\abs{\rho}^{N+n}} \leq \abs{\int_{N_\epsilon} e^{-x\cdot\rho} P^N(x) dx } \leq \frac{C}{\abs{\rho}^{N+n+1/2}}
\end{equation}
  for all large  \(|\rho|\) and the theorem is proved.
\end{proof}
It remains to prove Theorem \ref{cgolp} and Theorem
\ref{LTharmonic}, which are the subjects of Section
\ref{sec:besovestimates} and Section \ref{sectionPolynomial}.

\section{Estimates for Fundamental Solutions}
\label{sectionFive}

The proof of Theorem \ref{cgolp} will rely on an
estimate of the solution to 

\begin{equation}\label{eq:12}
P_{\rho}(D)\psi := (\Delta -2 \rho \cdot \nabla) \psi = f.
\end{equation}

Although we will work in different norms, we follow the
outline in \cite{RuizNotes} and begin by
 estimating the convolution
\(||\chi_{\varepsilon}*g||_{L^{\infty}}\) where  \(\chi\)
is a Schwartz class function, 
\begin{eqnarray}\nn
   g(\xi) = \frac{1}{P_{\rho}(\xi)}
 \qquad\text{and}\qquad
 \chi_{\varepsilon}(\xi) =
 \frac{1}{\varepsilon^{n}}\chi\left(\frac{\xi}{\varepsilon}\right)
 \end{eqnarray}
We will prove these estimates for a fairly general \(P\),
using a geometric approach
similar to that in
\cite{AgmonHormander}. The key properties of the
symbol \(P(\xi)\) are the codimension of its
characteristic variety (the set \(\M = P^{-1}(0)\)) and
the order to which it vanishes as
\(\xi\rightarrow\M\). The dimension of \(\M\) tells us the
behavior of the solutions to the homogeneous differential
equation, while the order of vanishing tells us the
behavior of the the particular solutions \(G*f\). In the
case of equation (\ref{eq:12}), the codimension is 2 and
\(P\) vanishes simply on \(\M\).

\begin{theorem}
\label{RuizEst}
Suppose that $\chi(x) \in \mathscr{S}(\R^n,\C)$ and $\chi_\varepsilon (x) := \varepsilon^{-n} \chi(\frac{x}{\varepsilon})$. If $P(\xi)$ satisfies
\begin{enumerate}[i)]
\item $P: \R^n \to \R^k$ is smooth
\item $\M = P^{-1}\{0\}$ is compact
\item $DP_{|\M}$ has constant rank, and
\item $\mathop{\liminf}\limits_{\abs{\xi} \to \infty} \abs{P} \geq B > 0$
\end{enumerate}
then
\begin{enumerate}[a)]
\item $\M$ is a smooth embedded codimension $k$ manifold in $\R^n$
\item $\norm{ \chi_\varepsilon \ast \delta_{\M}}_{L^\infty} \leq \frac{C}{\varepsilon^k}$
\item\label{item:1} If $P$ is real or complex valued ($k=1$ or $2$), then
\begin{equation}\nn
\norm{\chi_\varepsilon \ast \frac{1}{P}}_{L^\infty} \leq \frac{C}{\varepsilon}.
\end{equation}
Moreover, if $k \geq 2$ and $F$ is a complex valued function satisfying $\abs{F(P)} \leq \frac{1}{\abs{P}}$ then
\begin{equation}\nn
\norm{\chi_\varepsilon \ast F(P)}_{L^\infty} \leq \frac{C}{\varepsilon}.
\end{equation}
\end{enumerate}
\end{theorem}
\begin{remark}
We define
\begin{equation}\nn
\langle \delta_\M, \phi \rangle := \int_\M \phi d\sigma_\M
\end{equation}
where $d\sigma_\M$ is the natural element of surface area on $\M$.
\end{remark}
\begin{remark}
\label{PV}
If $k\geq 2$ then $\frac{1}{P} \in L^1_{loc}$ is a well
defined distribution on the whole of $\R^n$. If $k=1$ we
will use the principal value
\begin{equation}\nn
\langle \frac{1}{P}, \phi \rangle := \int_{N_\delta(\M)} \left( \phi(y) - \phi(m(y)) \right) \frac{dy}{P(y)} + \int_{\R^n \setminus N_\delta(\M)} \phi(y) \frac{dy}{P(y)},
\end{equation}
where $N_\delta(\M)$ is a neighborhood of  \(\M\)  and
$m(y)$ associates with each  \(y\in N_\delta(\M)\) the
closest point in  \(\M\). Both of which are described more
explicitly in the proposition below. 
\end{remark}

The following proposition recalls some immediate
consequences of the implicit function theorem. We don't
include a proof.
\begin{proposition}
\label{IFT}
Suppose that i), ii) and iii) in Theorem \ref{RuizEst} are satisfied. Then
\begin{enumerate}[A)]
\item $DP|_\M$ has full rank $k$
\item $\M$ is a smooth compact embedded submanifold of $\R^n$
\item $\exists\ \delta > 0$ and a Lipschitz constant $L_\delta$ such that writing
\begin{equation}\nn
N_\delta(\M) = \{ x \in \R^n \mid d(x,\M) \leq \delta \},
\end{equation}
every $x \in N_\delta(\M)$ has a unique closest point $m(x)$ in $\M$. The map
\begin{equation}\nn
\eta: N_\delta(\M) \to \M \times B_\delta^k(0)
\end{equation}
defined by
\begin{equation}\label{eq:normalcoords}
\eta(x) = \left( m(x), \abs{x - m(x)} \frac{DP_{m(x)} \left( x - m(x) \right)}{\abs{DP_{m(x)} \left(x-m(x)\right)}}\right)
\end{equation}
is a global diffeomorphism from $N_\delta(\M)$ onto $\M \times B_\delta^k(0)$. Both $\eta$ and $\eta^{-1}$ are Lipshitz with uniform constant $L_\delta$.
\item Every point $m \in \M$ has a $\delta$-neighborhood $U_\delta(m) \subset \M$ that is diffeomorphic to a ball in $\R^{n-k}$, i.e.
\begin{equation}\nn
\psi_m : U_\delta(m) := B_\delta^n(m) \cap \M \to B_\delta^{n-k}(0).
\end{equation}
Both $\psi_m$ and $\psi_m^{-1}$ are Lipschitz with uniform constant $L_\delta$.
\end{enumerate}
\end{proposition}

Two corollaries (also stated without proof) are:
\begin{corollary}
For $x\in \R^n$,
\begin{equation}\nn
\operatorname{Area}\left( B_r^n(x) \cap \M \right) := \int_{\M\cap B_r^n(x)} d\sigma_\M \leq C_\delta r^{n-k}
\end{equation}
\end{corollary}
\begin{corollary}
For $x \in N_\delta(\M)$,
\begin{equation}\nn
\abs{P(x)} \geq C_\delta d(x,\M).
\end{equation}
\end{corollary}

\medskip
We are going to use diffeomorphisms to rewrite integrals over manifolds as integrals over Euclidean balls, where we can do some explicit calculations. Since our integrals will involve convolutions with Schwartz class functions, we need to describe the properties that the pullbacks of such functions inherit.

\begin{definition}
\label{eMollDef}
A family of $\varepsilon$-mollifiers, $\chi_\varepsilon(x,y)$, defined on $\O_1 \times \O_2 \subset \R^n \times \R^n$ satisfies
\begin{enumerate}[i)]
\item $\sup_{x\in\O_1} \int_{\O_2} \abs{\chi_\varepsilon (x,y)}dy \leq C$
\item $\abs{\chi_\varepsilon(x,y)} \leq \frac{C_N}{\varepsilon^n} \left(\frac{\varepsilon}{\abs{x-y}} \right)^N$ for all $N \in\N$
\item $\abs{\nabla_y \chi_\varepsilon(x,y)} \leq \frac{C_N}{\varepsilon^{n+1}} \left( \frac{\varepsilon}{\abs{x-y}} \right)^N$ for all $N\in\N$
\end{enumerate}
\end{definition}

\begin{lemma}
If $\chi\in \mathscr{S}$, then
\begin{equation}\nn
\chi_\varepsilon (x,y) := \frac{1}{\varepsilon^n} \chi\left(\frac{x-y}{\varepsilon}\right)
\end{equation}
is a family of $\varepsilon$-mollifiers defined on $\O_1 \times \O_2 = \R^n \times \R^n$.
\end{lemma}

\begin{definition}
The pullback of a family of $\varepsilon$-mollifiers is defined\footnote{It seems natural to include a factor of $\operatorname{det}(D\psi)$ in \eqref{psiPB}, to treat $\chi_\varepsilon dx_1 \wedge \cdots \wedge dx_n$ as an $n$-form. We don't add the factor because it makes the proof of Lemma \ref{lemmaEmollPB} slightly longer.} to be
\begin{equation}
\label{psiPB}
\psi^*\chi_\varepsilon (x,y) := \chi_\varepsilon(\psi(x), \psi(y)).
\end{equation}
\end{definition}

The next lemma explains why we need to work with general $\varepsilon$-mollifiers.
\begin{lemma}
\label{lemmaEmollPB}
If $\psi$ and $\psi^{-1}$ are uniformly Lipschitz
diffeomorphisms, then the pullback of a family of
$\varepsilon$-mollifiers is a family of
$\varepsilon$-mollifiers.
\end{lemma}
\begin{proof}
Let $L_1$ and $L_2$ be the Lipschitz constants for $\psi$ and $\psi^{-1}$, respectively. For i), we estimate
\begin{eqnarray}\nn
\sup_{x\in \psi^{-1}(\O_1)} \int_{\psi^{-1}(\O_2)}
\chi_\varepsilon(\psi(x),\psi(y)) dy &=& \sup_{x\in \O_1}
\int_{\O_2} \chi_\varepsilon(x,y)
\frac{dy}{\operatorname{det}(D\psi(y))}
\\\nn
&\leq& \sup_{x\in\O_1} L_2^n \int_{\O_2} \chi_\varepsilon(x,y) dy
\end{eqnarray}
Next
\begin{equation}\nn
\abs{\chi_\varepsilon (\psi(x), \psi(y))} \leq \frac{C_N}{\varepsilon^n} \left(\frac{\varepsilon}{\abs{\psi(x) - \psi(y)}} \right)^N \leq \frac{C_N L_2^N}{\varepsilon^n} \left( \frac{\varepsilon}{\abs{x-y}} \right)^N.
\end{equation}
Finally, for iii),
\begin{multline}\nn
\abs{\nabla_y \chi_\varepsilon(\psi(x),\psi(y))} = \abs{ D\psi \cdot \nabla_v \chi_\varepsilon(u,v)}_{\big|\substack{u=\psi(x)\\v=\psi(y)}} \\
\leq L_1 \frac{C_N}{\varepsilon^{n+1}} \left( \frac{\varepsilon}{\abs{\psi(x) - \psi(y)}} \right)^N \leq \frac{C_N L_1 L_2^N}{\varepsilon^{n+1}} \left( \frac{\varepsilon}{\abs{x-y}} \right)^N
\end{multline}
\end{proof}

\begin{proposition}
\label{theorem1}
Let $\chi_\varepsilon$ be a family of
$\varepsilon$-mollifiers defined on $\O_1 \times \O_2
\subset \R^n \times \R^n$ and $\M$ a compact embedded
submanifold of $\R^n$ of codimension $k$. Then
\begin{equation}\nn
  \sup_{x\in\O_1} \int_{\M\cap\O_2} \abs{ \chi_\varepsilon(x,m) d\sigma_\M(m)} \leq \frac{C}{\varepsilon^k}
\end{equation}
for small $\varepsilon$.
\end{proposition}
\begin{proof}
  We may assume that $\M \subset \O_2$. Let $\delta$ be
  the uniform constant in Proposition \ref{IFT}. Fix
  $x\in\O_1$ and assume that $\varepsilon <
  \delta$. According to ii) in Definition \ref{eMollDef}
  we have
\begin{equation}\nn
\abs{ \int_{\M \cap \{ m \mid \abs{x-m} \geq \delta \}} \chi_\varepsilon(x,m) d\sigma_\M (m) } \leq \frac{C_N}{\varepsilon^n} \left( \frac{\varepsilon}{\delta} \right)^N \operatorname{area}(\M).
\end{equation}
On the other hand
\begin{eqnarray}\nn
  \abs{ \int_{\M \cap \{m \mid \abs{x-m} \leq \varepsilon\}}
    \chi_\varepsilon (x,m) d\sigma_\M(m) } 
  &\leq& \frac{C_0}{\varepsilon^n}
\operatorname{area} \left(\M \cap B_\varepsilon^n(x)\right) \\\nn
  &\leq& \frac{C_0}{\varepsilon^n} L_\delta^{n-k} \varepsilon^{n-k} = \frac{C_0 L_\delta^{n-k}}{\varepsilon^k},
\end{eqnarray}
where $L_\delta$ is the Lipschitz constant. To estimate
the remaining part of the integral, we use local
coordinates $\psi$, based at $m(x)$, the point on $\M$
closest to $x$, as described in Proposition \ref{IFT}
D). Let $\Psi = \psi^{-1}$. Then
\begin{equation}\nn
\abs{ \int_{\M \cap \{ m \mid \varepsilon < \abs{x-m} < \delta \}} \chi_\varepsilon(x,m) d\sigma_\M(m) } = \abs{ \int_{B_\delta^{n-k}(0) \setminus B_\varepsilon^{n-k}(0)} \Psi^*\chi_\varepsilon \Psi^*d\sigma_\M}.
\end{equation}
Because
\begin{eqnarray*}
\abs{\chi_\varepsilon(x,m)} &\leq&
\frac{C_N}{\varepsilon^n} \left( \frac{\varepsilon}{
    (\abs{x-m(x)}^2 + \abs{m(x) - m}^2)^{1/2}} \right)^N 
\\
&\leq& \frac{C_N}{\varepsilon^n} \left( \frac{\varepsilon}{\abs{m(x) - m}} \right)^N = \frac{C_N}{\varepsilon^n} \left( \frac{\varepsilon}{\rho} \right)^N,
\end{eqnarray*}
where $\rho = \abs{m(x) - m}$, we may use polar coordinates centered at $m(x)$ to see that
\begin{eqnarray}\nn
\int_{B_\delta^{n-k}(0) \setminus B_\varepsilon^{n-k}(0)}
\abs{ \Psi^*\chi_\varepsilon \Psi^*d\sigma_\M} 
&\leq&
L_\delta^{n-k} \!\! \int_{S^{n-k-1}}
\!\!\!\!\!\!\!\!\!\!\!\!\!\!d \sigma_{S^{n-k-1}} \!
\int_\varepsilon^\delta \frac{C_N}{\varepsilon^n} \left(
  \frac{\varepsilon}{\rho} \right)^N \rho^{n-k-1} d\rho
\\\nn
&\leq& L_\delta^{n-k} \omega_{n-k-1} C_N \varepsilon^{N-n}
\frac{\abs{\delta^{n-k-N} -
    \varepsilon^{n-k-N}}}{\abs{n-k-N}}
\\\nn
&\leq& L_\delta^{n-k} \omega_{n-k-1} C_N \frac{\varepsilon^{-k}}{\abs{n-k-N}}
\end{eqnarray}
where $S_{n-k-1}$ is the unit sphere in $\R^{n-k}$ and $\omega_{n-k-1}$ its surface measure. The claim follows by taking $N > n-k$.
\end{proof}

\begin{remark}
\label{notQuiteEmoll}
In the proof of Proposition \ref{theorem1}, when considering $x\in N_\delta(\M)$, we only required that the mollifier satisfy
\begin{equation}
\nn
\abs{\chi_\varepsilon(x,y)} \leq \frac{C_N}{\varepsilon^n} \left( \frac{\varepsilon}{\abs{m(x)-m(y)}} \right)^N.
\end{equation}
We will use this observation in the proof of Proposition
\ref{theorem2} below, which will finish the proof of
Theorem \ref{RuizEst}.
\end{remark}

\begin{proposition}
\label{theorem2}
Let $\chi_\varepsilon$ be a family of
$\varepsilon$-mollifiers, 
$\M$, $P$ and $k \geq 2$ satisfy the conditions in Theorem
\ref{RuizEst}, and $F:\R^k \to \C$ satisfy $\abs{F(P)}
\leq \frac{C}{\abs{P}}$. Then, for sufficiently small $\varepsilon$,
\begin{equation}\nn
\abs{\int_{\R^n} \chi_\varepsilon(x,y) F(P(y)) dy} \leq \frac{C}{\varepsilon}.
\end{equation}
If $k=1$ then
\begin{equation}\nn
\abs{\int_{\R^n} \frac{\chi_\varepsilon(x,y)}{P(y)} dy} \leq \frac{C}{\varepsilon},
\end{equation}
where $\frac{1}{P}$ is defined by principal value as in Remark \ref{PV}.
\end{proposition}
\begin{proof}
We assume that $\varepsilon < \frac{\delta}{2}$, with $\delta$ the constant in Proposition \ref{IFT} C). Because $\abs{F(P)} \leq \frac{C}{\abs{P}} \leq \frac{C_\delta}{\varepsilon}$ on $N_\delta(\M) \setminus N_\varepsilon(\M)$ and $\leq C_\delta$ outside $N_\delta(\M)$,
\begin{equation}\nn
\int_{\R^n \setminus N_\varepsilon(\M)} \abs{\chi_\varepsilon F(P) } dy \leq \sup_{y \in \R^n \setminus N_\varepsilon(\M)} \abs{ F(P(y)) } \norm{\chi_\varepsilon}_{L^1} \leq \frac{C}{\varepsilon}.
\end{equation}

For the moment, we restrict to the case that $k= \operatorname{codim}(\M) \geq 2$, so that $F(P) \in L^1(\R^n)$. If $x \notin N_\delta(\M)$, then
\begin{equation}\nn
\sup_{y\in N_\epsilon(\M)} \abs{\chi_\varepsilon (x, y)} \leq \left( \frac{\varepsilon}{\delta/2} \right)^N \frac{C_N}{\varepsilon^n}
\end{equation}
so that
\begin{equation}\nn
\sup_{x \notin N_\delta(\M)} \int_{N_\varepsilon(\M)} \abs{\chi_\varepsilon F(P)} dy \leq \int_{N_\varepsilon (\M)} \abs{F(P)} dy \left( \frac{\varepsilon}{\delta/2} \right)^N \frac{C_N}{\varepsilon^n}.
\end{equation}
and choosing $N \geq n-1$ shows that this bounded by a
constant over  \(\varepsilon\).

If $x \in N_\delta(\M)$, we can use the diffeomorphism
$\eta$ and its inverse $H$, described in C) of Proposition
\ref{IFT} to obtain
\begin{multline}\nn
  \sup_{x\in N_\delta(\M)} \int_{N_\varepsilon (\M)} \abs{\chi_\varepsilon(x,y) F(P(y)) } dy \\
  = \sup_{u \in \M \times B_\delta^k(0)} \int_{\M \times
    B_\varepsilon^k(0)} \abs{H^*\chi_\varepsilon(u,v)
    F(P(H(v))) } \frac{d\sigma_\M (m)
    ds}{\abs{\operatorname{det}(D\eta)}}
\end{multline}
where $v = (m, s) \in \M \times
B_\varepsilon^k(0)$. Because $\abs{F(P(H(s)))} \leq
\frac{C}{\abs{P(y)}} \leq \frac{C}{\abs{s}}$ here and
$\abs{\operatorname{det}(D\eta)}$ is bounded from below by
the the  \(n\)-th power of the Lipschitz constant
\(L_{2}\), this is bounded by  
\begin{equation}
\label{eq1}
\leq CL_2^{-n} \int_{B_\varepsilon^k(0)} \left( \sup_{u \in \M \times B_\delta^k(0)} \int_\M \abs{H^*\chi_\varepsilon} d\sigma_\M \right) \frac{1}{\abs{s}} ds.
\end{equation}
For each fixed $s$,
\begin{equation}\nn
\abs{H^*\chi_\varepsilon} \leq \frac{C_N}{\varepsilon^n} \left( \frac{\varepsilon}{\abs{u - (m,s)}} \right)^N,
\end{equation}
so according to Remark \ref{notQuiteEmoll} we can apply Proposition \ref{theorem1} to the manifold $\M \times \{s\}$ to show that the quantity in brackets in \eqref{eq1} satisfies
\begin{equation}\nn
\sup_{u \in \M \times B_\delta^k(0)} \int_\M \abs{H^* \chi_\varepsilon} d\sigma_\M \leq \frac{C}{\varepsilon^k}.
\end{equation}
This implies the estimate
\begin{equation}\nn
\sup_{x\in N_\delta(\M)} \int_{N_\varepsilon(\M)} \abs{ \chi_\varepsilon(x,y) F(P(y))} dy \leq \int_{B_\varepsilon^k(0)} \frac{C}{\varepsilon^k} \frac{ds}{\abs{s}} = \frac{C}{\varepsilon^k} \cdot \varepsilon^{k-1},
\end{equation}
which completes the proof in the codimension $2$ case.

\medskip
If $\M$ is of codimension one we have the definition
\begin{equation}\nn
\langle \frac{1}{P}, \phi \rangle = \int_{N_\delta(\M)} \left( \phi(y) - \phi(m(y)) \right) \frac{dy}{P(y)} + \int_{\R^n \setminus N_\delta(\M)} \phi(y) \frac{dy}{P(y)}
\end{equation}
and note that this agrees with $\int_{\R^n} \phi \frac{dy}{P}$ for all $\phi \in C^\infty_0(\R^n \setminus \M)$. With this definition,
\begin{equation}\nn
\frac{1}{P} \ast \chi_\varepsilon = \int_{\R^n\setminus N_\varepsilon} \chi_\varepsilon \frac{dy}{P} + \int_{N_\varepsilon(\M)} \frac{ \left( \chi_\varepsilon (x,y) - \chi_\varepsilon(x,m(y)) \right)}{P(y)} dy.
\end{equation}
We estimate the first integral as we did  in the
codimension $\geq 2$ case, and  rewrite the second as
\begin{equation}\nn \int_\M \left[
\int_{-\varepsilon}^\varepsilon \frac{
\chi_\varepsilon(m(x), \nu(x), m(y), \nu(y)) -
\chi_\varepsilon(m(x),\nu(x),m(y),0)}{P(m(y), \nu,y)} d\nu(y)
\right] d\sigma_\M.
\end{equation}
where  \(m(x)\) again denotes the closest point on
\(\M\), and  \(\nu(x)\) are the normal coordinates, given
explicitly by the second component on the right hand side
of equation (\ref{eq:normalcoords}).
If we call the integral in brackets $\widetilde{\chi_\varepsilon}$, we see that
\begin{equation}\nn
\abs{\widetilde{\chi_\varepsilon}} \leq \frac{C_N}{\varepsilon^n} \abs{\frac{\varepsilon}{\abs{m(x) - m(y)}} }^N
\end{equation}
so that Remark \ref{notQuiteEmoll} applies here, and we may conclude that $\abs{\int_\M \widetilde{\chi_\varepsilon} d\sigma_\M} \leq \frac{C}{\varepsilon^k}$ with $k=1$ in this case.
\end{proof}

We need only one application of Theorem \ref{RuizEst} for our
proof of Theorem \ref{cgolp}. We return to (\ref{eq:12}) and set  

\begin{equation}
  \nn
  g(\xi) = \frac{1}{-\xi\cdot\xi - 2i \rho\cdot\xi}
\end{equation}

\begin{proposition}\label{sec:estim-source} There is a constant C, depending only
  on the dimension  \(n\) and \(\chi\in \mathscr{S}(\R^n,\C)\), so that
  \begin{equation}
    \label{eq:14}
    \norm{\chi_{\epsilon}*g}_{\infty} \le \frac{C}{\epsilon\rho}
  \end{equation}
\end{proposition}
\begin{proof}
  Let  \(\rho=s\Theta\) where  \(\Theta\in\C^{n}\) has
  unit norm and  \(s=|\rho|\) . We will apply the estimate in item \ref{item:1} from Theorem
  \ref{RuizEst}, but first we need
   do some scaling
  \begin{eqnarray}
    \nn
    \chi_\epsilon * g(s\eta) &=& -\int
    \chi\left(\frac{s\eta-\xi}{\epsilon}\right)
    \frac{1}{\xi\cdot\xi+2i s\Theta\cdot\xi}\frac{d^{n}\xi}{\epsilon^{n}}
    \\\na{letting  \(\sigma=s\xi\) gives}\nn
    &=&-\frac{1}{s^{2}} \int \chi\left(\frac{\eta-\sigma}{\frac{\epsilon}{s}}\right)
    \frac{1}{\sigma\cdot\sigma+2i\Theta\cdot\sigma}\frac{d^{n}\sigma}{(\frac{\epsilon}{s})^n} 
    \\\nn
    &=&\frac{1}{s^{2}}\chi_{\frac{\epsilon}{s}}*\frac{1}{\tilde{P}}
    \\\na{where  \(\tilde{P}= -\xi\cdot\xi-2i\Theta\cdot\xi\). According to Theorem \ref{RuizEst},\espace}\nn
    \norm{\chi_\epsilon * g(s\eta)}_{\infty}&\le&\frac{1}{s^{2}}\frac{C}{\frac{\epsilon}{s}}
    \le\frac{C}{s\epsilon}
    \\\na{Recalling that \(s = |\rho|\),  
      and that  \(\norm{\chi_\epsilon * g(s\eta)}_{\infty} =
      \norm{\chi_\epsilon * g}_{\infty}\) gives}\nn
    \norm{\chi_\epsilon*g}_{\infty}&\le&\frac{C}{\epsilon|\rho|}
   \end{eqnarray}
\end{proof}

\section{Proof of Theorem \ref{cgolp}}
\label{sec:besovestimates}

In order to prove Theorem \ref{cgolp}, we insert the
ansatz \eqref{eq:1} into \eqref{mFree} to see that
\(\psi\) must satisfy
\begin{eqnarray}
\label{eq:20}
(\Delta -2 \rho \cdot \nabla) \psi = -k^{2}(1-m)(1+\psi) \quad\text{in } D.
\end{eqnarray}
We replace the right hand side of (\ref{eq:20}) using
\begin{equation}
\nn
  Q = -k^{2}(1-m)\Phi_{D}
\end{equation}
 where  \(\Phi_{D}\) is smooth, compactly supported, and
 identically equal to one on the bounded domain  \(D\).
 We seek  \(\psi\) satisfying
 \begin{eqnarray}
\label{eq:21}
(\Delta -2 \rho \cdot \nabla) \psi =
Q(1+\psi)\qquad\text{in } \R^{n}
\end{eqnarray}
 noting that a solution to (\ref{eq:21}) in  \(\R^{n}\)
 will satisfy (\ref{eq:20}) in  an open neighborhood of \(D\). We will construct
 \(\psi\) by summing the series
 \begin{eqnarray}
   \label{eq:28}
   \psi = \Sum_{N=0}^{\infty} \psi^{N}
\\\na{where  \(\psi^{0}=0\) and  the remaining
  \(\psi^{N}\) satisfy\espace }
\label{eq:23}
(\Delta -2 \rho \cdot \nabla) \psi^{N} =
Q\psi^{N-1}
 \end{eqnarray}
 The existence of solutions to (\ref{eq:23}) and the
 convergence of the sum will follow from an estimate of
 solutions to the constant coefficient differential equation

\begin{equation}\label{eq:8}
P_{\rho}(D)\psi := (\Delta -2 \rho \cdot \nabla) \psi = f.
\end{equation}
The simplest estimate would follow from taking the
Fourier transform of both sides and dividing by the symbol
\(P_{\rho}(\xi)\). If we use the letter  \(g(\xi)\) to denote the
reciprocal of  \(P_{\rho}\), we want to estimate

\begin{eqnarray}
\nn
  \widehat{\psi} &=& g \widehat{f}
\\\noalign{\noindent or equivalently}\label{eq:10}
\psi &=& G*f
\end{eqnarray}
where \(*\) denotes convolution and \(G\) is the inverse
Fourier transform of \((2\pi)^{-n/2}g\).
A simple \(L^{\infty}\) estimate for \(g\) does not
hold because of the zeros of \(P\), but these affect the
behavior of \(\psi\) for large \(x\), and our goal is to
prove a strong local estimate. We are willing to prove an
estimate that allows \(\psi\) to grow as
\(x\rightarrow\infty\) in exchange for a good local
estimate, i.e.  \(L^{q}\) for large \(q\) on compact
sets. We will separate the local and global behavior by
writing \(G\), the inverse Fourier transform of \(g\), as
a sum of functions \(G_{j}\) with compact support, and
estimating each separately.\\

We introduce a dyadic partition
of unity. Let 
\begin{equation}\label{eq:32}
1 = \phi_0(s) + \sum_{j=1}^\infty \phi_j(s)
\end{equation}
where $\phi_0$ and $\phi$ are $C^\infty$ even functions of $s\in\R$, and
\begin{gather}\nn
\supp \phi_0 \subset [-2,2]
\\\nn
\supp \phi \subset [\tfrac{1}{2}, 2] 
\\\na{and}\nn
\Phi_j(x) := \phi(\tfrac{\abs{x}}{2^j}) \quad \text{for } j \geq 1, \quad \Phi_0(x) = \phi_0(\abs{x}),
\\\na{so that}\nn
\supp \Phi_j \subset B_{2^{j+1}}(0) \setminus B_{2^{j-1}}(0), \quad \supp \Phi_0 \subset B_2(0).
\end{gather}
We will make use of the fact that
\begin{equation}
  \nn
  \hat{\Phi_{j}}(\xi) = 2^{nj}\hat{\Phi}(2^{j}\xi) =
\frac{\hat{\Phi}(\frac{\xi}{\epsilon})}{\epsilon^{n}}
\end{equation}
which makes the  \(\{\hat{\Phi_{j}}\}\)  a family of  \(\epsilon\)-mollifiers with
 \(\chi=\hat{\Phi}\) and with \(\epsilon=2^{-j}\).\\

We expand  \(\psi\),  \(G\), and  \(f\) with respect
to this partition, i.e

\begin{eqnarray*}
  \psi = \sum \psi_{j} = \sum \Phi_{j}\psi
\\
f  = \sum f_{j} = \sum \Phi_{j}f
\\
G  = \sum G_{j} = \sum \Phi_{j}G
\end{eqnarray*}
so that (\ref{eq:10}) becomes

\begin{eqnarray}
\label{eq:11}
  \psi_{m} = \Phi_{m}\Sum_{k=0}^{\infty}\Sum_{j=0}^{\infty}G_{j}*f_{k}
\end{eqnarray}
If we recall that the support of the
convolution is a subset of the sum of  the
supports, we see that if \(r_1 = 2^{k-1}-2^{j+1} >0 \) or
\(r_2 = 2^{j-1}-2^{k+1} >0 \), the support of \(G_{j}*f_{k}\) is
contained outside the ball of radius \(r_1\) or \(r_2\), respectively. In particular,
this means that 
\begin{eqnarray}
\nn
  \Phi_{m}G_{j}*f_{k} = 0
\\\na{if}\nn
2^{m+1}  < 2^{j-1}- 2^{k+1}
\\\na{which will always be the case if}\nn
j> 3+ \max(k,m)
\end{eqnarray}
so that the second sum in (\ref{eq:11}) is finite
\begin{eqnarray}
  \nn
\psi_{m} = \Phi_{m}\Sum_{k=0}^{\infty}\Sum_{j=0}^{\max(k,m)+3}G_{j}*f_{k}.
\\\na{Taking the Fourier transform gives}\nn
\hat{\psi_{m}}=(2\pi)^{-n}\hat{\Phi_{m}}*\Sum_{k=0}^{\infty}\Sum_{j=0}^{\max(k,m)+3}g_{j}\hat{f_{k}}
\\\na{where  \(g_{j}=\hat{\Phi_{j}}*g = (2\pi)^{n} \hat{G_j}\), so that\espace}\nn
\norm{\hat{\psi_{m}}}_{p} \le (2\pi)^{-n}\norm{\hat{\Phi_{m}}}_{1}
\Sum_{k=0}^{\infty}\Sum_{j=0}^{\max(k,m)+3}\norm{g_{j}}_{\infty}
\norm{\hat{f_{k}}}_{p}
\end{eqnarray}

We may now 
estimate the convolution
\(||\hat{\Phi_{j}}*g||_{L^{\infty}}\) using  (\ref{eq:14})
of Proposition \ref{sec:estim-source} with \(\chi=\Phi\)
and \(\epsilon=2^{-j}\) to establish that
\begin{eqnarray}
  \nn
||g_{j}||_{\infty}\le \frac{C}{|\rho|}2^{j} 
\end{eqnarray}
for \(|\rho|\) sufficiently large, so that
\begin{eqnarray}\nn
\norm{\hat{\psi_{m}}}_{p} \le \norm{\hat{\Phi_{m}}}_{1}
\Sum_{k=0}^{\infty}\Sum_{j=0}^{\max(k,m)+3}\frac{C}{|\rho|}2^{j}
\norm{\hat{f_{k}}}_{p}
\end{eqnarray}

Because \(\hat{\Phi_{m}}(\xi) = 2^{nm}\hat{\Phi}(2^{m}\xi)\),
 its  \(L^{1}\) norm is the same as the  \(L^{1}\) norm
of  \(\hat{\Phi}\), which doesn't depend on  \(m\), so
\begin{eqnarray}\nn
\norm{\hat{\psi_{m}}}_{p} &\le& \frac{C}{|\rho|}
\Sum_{k=0}^{\infty}\left(\Sum_{j=0}^{\max(k,m)+3}2^{j}\right)
\norm{\hat{f_{k}}}_{p}
\\\nn
&\le&
\frac{C}{|\rho|}\Sum_{k=0}^{\infty}2^{\max(k,m)+4}\norm{\hat{f_{k}}}_{p}
\\\na{which we rewrite as}\label{eq:18}
\Sup_{m}2^{-m}\norm{\hat{\psi_{m}}}_{p} &\le&\frac{C}{|\rho|}\Sum_{k=0}^{\infty}2^{k}\norm{\hat{f_{k}}}_{p}
\end{eqnarray}
with a new constant  \(C\) that is  \(2^{4}\) times the
old one.\\

Our goal is to estimate the  \(L^{q}(D)\) norm of
\(\psi\) on a compact set \(D\) for \(q>2\), and this is
bounded by the left hand side of (\ref{eq:18}) if we
choose  \(p<2\) to be the dual exponent. In our
application,  \(f\) will be the right hand side of
(\ref{eq:23}) which will have its support in  \(D\), so
the sum on the right hand side of (\ref{eq:18}) will also
be a finite sum,
bounded by a constant times  \(\norm{\hat{f}}_{p}\).  We
we will have the desired bound for \(\psi\) as long as we
can guarantee that  the Fourier transform of \(f\) is
in  \(L^{p}\) for all \(p\le 2\).\\

In the special case that  \(p=2\), the Plancherel
inequality tells us that (\ref{eq:18}) is equivalent to
\begin{equation}
  \nn
  \Sup_{m}\ 2^{-m}\norm{\psi_{m}}_{2} \le\frac{C}{|\rho|}\Sum_{k=0}^{\infty}2^{k}\norm{f_{k}}_{2}
\end{equation}
This kind of estimate was used in \cite{AgmonHormander} to
study constant coefficient PDE's with simple
characteristics, including, as the principal example, the
free Helmholtz equation. The norms defined there were:
\begin{eqnarray}\nn
\norm{f}_{B_{k/2}^*} &:=& \sup_{0 \leq j < \infty}
\frac{1}{2^{jk/2}} \norm{f_{j}}_{2} = \lVert\hat{f}\rVert_{B^{-k/2}_{2,\infty}}
\\\nn
\norm{f}_{B_{k/2}} &:=& \sum_{j=0}^\infty 2^{jk/2} \norm{f_{j}}_{2} = \lVert\hat{f}\rVert_{B^{k/2}_{2,1}}
\end{eqnarray}
The authors showed, in particular, that the incident waves
for the Helmholtz equation 
in $B_{1/2}^*$ were exactly the Herglotz wave
functions.

Our estimate in (\ref{eq:18}) may be written
as

\begin{equation}
  \label{eq:19}
  \norm{\psi}_{\hat{B^{-1}_{p,\infty}}}
\le\frac{C}{|\rho|}\norm{f}_{\hat{B^1_{p,1}}}
\end{equation}
with
\begin{eqnarray}\nn
\norm{\psi}_{\hat{B^{-1}_{p,\infty}}} :=
\sup_{0 \leq j < \infty} \frac{1}{2^{j}}
\norm{\hat{\psi_j}}_p
\\\nn
\norm{f}_{\hat{B^1_{p,1}}} := \sum_{j=0}^\infty 2^{j} \norm{\hat{f_j}}_p
\end{eqnarray}

In this notation we have proved\\ 
\begin{proposition}
  For every  \(f\in \hat{B^1_{p,1}}\), there exists a
  \(\psi\in \hat{B^{-1}_{p,\infty}}\) satisfying
  (\ref{eq:8}) and the
    estimate (\ref{eq:19}).
\end{proposition}

We now return to (\ref{eq:23}). We will show, in Lemma
\ref{sec:lemma-Q} below, that  

\begin{equation}
  \label{eq:22}
  Q\in \hat{B^1_{p,1}}
  \qquad\text{and}\qquad||Qg||_{\hat{B^1_{p,1}}}
\le C_{Q} ||g||_{\hat{B^{-1}_{p,\infty}}}
\end{equation}
where  \(C_{Q}\) denotes a constant depending on  \(Q\).  Combining (\ref{eq:22}) with (\ref{eq:19})
shows that
\begin{eqnarray}
  \nn
  \norm{\psi^{N}}_{\hat{B^{-1}_{p,\infty}}}
&\le&\frac{C_{Q}}{|\rho|}\norm{\psi^{N-1}}_{\hat{B^{-1}_{p,\infty}}}
\le\left(\frac{C_{Q}}{|\rho|}\right)^{N}\norm{Q}_{\hat{B^{1}_{p,1}}}
\end{eqnarray}
and hence that the series  (\ref{eq:28}) converges when
\(|\rho|>C_{Q}\) and thererefore that, for
\(2<q=\frac{p}{p-1}\) and  \(D\) contained in a ball of
radius  \(R\),    the sum  \(\psi\)
satisfies

\begin{eqnarray}
  \nn
  \norm{\psi}_{L^{q}(D)}\le R
  \norm{\psi}_{\hat{B^{-1}_{p,\infty}}}
\le \frac{C}{|\rho|}\norm{Q}_{\hat{B^{1}_{p,1}}}
\end{eqnarray}
and establishes \eqref{cgolpest} for  all \(q>2\) (and, because
\(D\) is bounded, for \(q<2 \) as well) and therefore
proves Theorem~\ref{cgolp}.\\

It remains to prove (\ref{eq:22}). The function  \(Q\)
satisfies
\begin{eqnarray}
  \nn
  Q = \prod\limits_{i=1}^{n}(H^{+}(x_{i}) - H^+(x_i-1))q(x)
\end{eqnarray}
where  \(q\) is smooth and supported in a ball of radius  \(R\), and
\(H^{+}(t)\) is the Heavyside function, the indicator
function of the positive half line.

\begin{eqnarray}
  \nn
  \norm{q}_{\hat{B^1_{p,1}}}  =
  \Sum_{j=0}^{\log_{2}R}2^{j}||\hat{\Phi_{j}}*\hat{q}||_{p}
  \le 2R\sup\limits_{j}\norm{\hat{\Phi_{j}}}_{1}||\hat{q}||_{p}
  = 2R\norm{\hat{\Phi}}_{1}||\hat{q}||_{p}
\end{eqnarray}
so  \(q\in \mathscr{F}B^1_{p,1}\). The lemma below tells us that
multiplication by the Heavyside function preserves
\(\mathscr{F}B^1_{p,1}\) and that multiplication by smooth
compactly supported  \(q\) maps
\(\mathscr{F}B^{-1}_{p,\infty}\) to  \(\mathscr{F}B^1_{p,1}\). This
is enough to establish (\ref{eq:22}) and finish this
section.\\

\begin{lemma}\label{sec:lemma-Q} Suppose that  \(q\) is smooth and supported
  in the ball of radius  \(R\), and  \(\Theta\) a unit
  vector in  \(\R^{n}\). Then   
  \begin{eqnarray}
    \label{eq:24}
    &\norm{qg}_{\hat{B^1_{p,1}}} \le
    2R^{2}\norm{\hat{q}}_{1}||g||_{\hat{B^{-1}_{p,\infty}}}
    \\\label{eq:25}
    &\norm{\big(H_{+}(x\cdot\Theta) - H^+(x\cdot\Theta - 1)\big)g(x)}_{\hat{B^1_{p,1}}} \le C_{p}\norm{g}_{\hat{B^1_{p,1}}}
  \end{eqnarray}
  for \(1<p<\infty\).
\end{lemma}
\begin{proof}
  \begin{eqnarray}
    \nn
    \norm{qg}_{\hat{B^1_{p,1}}} &=&
    \Sum_{j=0}^{\infty}2^{j}\norm{\hat{(qg\Phi_{j})}}_{p}
    \\\na{Because  \(q\) has compact support, the sum is
      finite, i.e. \espace}\nn
    &=&\Sum_{j=0}^{\log_{2}R}2^{j}\norm{\hat{q}*\hat{g_{j}}}_{p}
    \\\na{where  \(g_{j}\) denotes  \(g*\Phi_{j}\)}\nn
    &\le&
    \Sum_{j=0}^{\log_{2}R}2^{2j}\norm{\hat{q}}_{1}
    \left(2^{-j}\norm{\hat{g_{j}}}_{p}\right)
    \\\nn
    &\le&
    \left(\Sum_{j=0}^{\log_{2}R}2^{2j}\right)\norm{\hat{q}}_{1}\norm{g}_{\hat{B^{-1}_{p,\infty}}}
    \\\nn
    &\le&2R^{2}\norm{\hat{q}}_{1}\norm{g}_{\hat{B^{-1}_{p,\infty}}}
  \end{eqnarray}
which establishes (\ref{eq:24}). To prove (\ref{eq:25})
\begin{eqnarray}
  \nn
  \norm{H^{+}(x\cdot\Theta)g}_{\hat{B^1_{p,1}}}
  &=&\Sum_{j=0}^{\infty}2^{j}\norm{\hat{H^{+}g_{j}}}_{p}
\\\nn
&=&\Sum_{j=0}^{\infty}2^{j}\norm{\hat{H^{+}}*\hat{g_{j}}}_{p}
\\
\na{but convoution with  \(\hat{H^{+}}(\xi\cdot\Theta)\)
  is just a one dimensional Hilbert transform in the
  direction  \(\Theta\), which is bounded from  \(L^{p}\)
  to  \(L^{p}\) for all  \(1<p<\infty\), so that\espace}\nn
&\le&\Sum_{j=0}^{\infty}2^{j}C_{p}\norm{\hat{g_{j}}}_{p}
\\\nn
&\le&C_{p}\norm{g}_{\hat{B^1_{p,1}}}.
\end{eqnarray}
The same estimate holds for $H^+(x\cdot\Theta - 1)$
because rigid motions induce bounded maps from
\(\mathscr{F}B^1_{p,1}\) to itself.
\end{proof}

\section{Function spaces}
\label{sectionFunctionSpaces}
We have proved estimates using  \(\mathscr{F}B^1_{p,1}\) and
\(\mathscr{F}B^{-1}_{p,\infty}\) norms, but have, so far said
very little about the function spaces, other than pointing
out that smooth compactly supported functions times
Heavyside functions belong to 
\(\mathscr{F}B^1_{p,1}\), and that, for  \(p<2\) and  \(q\) the
dual exponent, any
\(u\in \mathscr{F}B^{-1}_{p,\infty}\)  is in  \(L^{q}\) of every
compact set.\\

These spaces are Fourier transforms of 
Besov-spaces, which  are defined in
\cite{BL,Stein,Triebel1} using the partition of unity  in (\ref{eq:32}).
For $s\in \R$, $0<p\leq \infty$ and $0<q\leq \infty$,
\begin{equation}\nn
B^s_{p,q} = B^s_{p,q}(\R^n) = \{ f \in \mathscr{S}'(\R^n) \mid \norm{f}_{B^s_{p,q}} < \infty \}
\end{equation}
where
\begin{equation}\nn
\norm{f}_{B^s_{p,q}} := \left( \sum_{j=0}^\infty \left( R_j^s \norm{\mathscr{F}^{-1}\!(\Phi_j \hat{f}\, )}_{L^p} \right)^q \right)^{1/q}
\end{equation}
with  \(R_{j}=2^{j}\), and  with the usual modification for $q=\infty$.

\begin{definition}
  We say that $f \in \BSF$ if $\hat{f} \in B^s_{p,q}$ and
  write $\norm{f}_{\BSF} = \lVert \hat{f}
  \rVert_{B^s_{p,q}}$. Note that
\begin{equation}\nn
\norm{f}_\BSF = \left( \sum_{j=0}^\infty \left( R_j^s \norm{\hat{f_j}}_{L^p} \right)^q\right)^{1/q} =  \left( \sum_{j=0}^\infty \left( R_j^s \norm{\hat{\Phi_j} \ast \hat{f}}_{L^p} \right)^q\right)^{1/q} .
\end{equation}
\end{definition}

\smallskip

The fact that \(\BSF\) is a Banach space follows from the
that corresponding fact for \( B^s_{p,q}\)
\cite[2.3.3]{Triebel1}. We simply note that the Fourier
transform, acting on tempered distributions, is one to one, and
that convergence in the \( B^s_{p,q}\) norm implies convergence
as tempered distributions.

\section{Proof of Theorem \ref{LTharmonic}}
\label{sectionPolynomial}
In this section, we will use what are sometimes
called array, or componentwise operations, as well as
standard multi-index notation. If $\eta$ is a vector in
$\C^n$, and $\alpha$ is a multi-index (i.e. also a
vector), we will use \(\eta^\alpha\) to
mean the product
\begin{equation}\nn
\eta^\alpha = \eta_1^{\alpha_1} \cdots \eta_n^{\alpha_n}
\end{equation}
When  \(a\) is a scalar,  \(\eta^{a}\) will denote the vector
\begin{equation}\nn
\eta^{a} = (\eta_1^{a}, \ldots, \eta_n^{a})
\end{equation}
Similarly, we will use a scalar divided by a vector, or a vector
divided by a vector, to denote
componentwise division, e.g.
\begin{eqnarray}\nn
\frac{1}{\eta} = \left( \frac{1}{\eta_1}, \ldots,
  \frac{1}{\eta_n} \right),
\end{eqnarray}
We let $\sigma_k(\eta)$ denotes the $k$'th elementary
symmetric function of $(\eta_1, \ldots, \eta_n)$. The two
symmetric functions
we will make use of are
\begin{gather}\nn
\sigma_n(\eta) = \prod_{i=1}^n \eta_i,\\\nn
\sigma_{n-1}(\eta) = \sum_{i=1}^{n} \prod_{j\neq i} \eta_j.
\end{gather}
In this section, we will use the superscript
$\hat{\phantom{i}}$ to indicate that an index does not occur, so that
\begin{equation}\nn
\eta_{\hat{i}} = (\eta_1, \ldots, \eta_{i-1}, \eta_{i+1}, \ldots, \eta_n)
\end{equation}
means the $n-1$-dimensional vector that omits the $i$'th
component of $\eta$. We will use the notation
$P_{\hat{i}}$ and $P(\eta_{\hat{i}})$ interchangeably to
denote a polynomial which does not depend on the $i$'th
variable.\\

We will also use the superscript $\hat{\phantom{i}}$ to
denote the Laplace transform, \(\hat{P}\), of a degree $N$
homogeneous polynomial $P(x)$, given by
\begin{equation}\label{eq:26}
\widehat{P}(\rho) = \int_{x>0} e^{-\rho\cdot x} P(x) dx
\end{equation}
where $x>0$ means that each component $x_i > 0$. Making
the substitutions  $y_i = \rho_i x_i$, with $\rho$ real and $\rho > 0$ (for the moment) we have
\begin{equation}\nn
\widehat{P}(\rho) = \int_{y>0} e^{-1\cdot y}P\left(\frac{y}{\rho}\right) \sigma_n\left(\frac{1}{\rho}\right) dy.
\end{equation}
If $P = \sum_{\abs{\alpha}=N} p_\alpha x^\alpha$, then
\begin{equation}\nn
\widehat{P}(\rho) = \sum_{\abs{\alpha}=N} p_\alpha \frac{1}{\rho^{\alpha+1}} \int_{y>0} e^{-1\cdot y}y^\alpha dy = \sum_{\abs{\alpha}=N} p_\alpha \left( \frac{1}{\rho} \right)^{\alpha+1} \alpha!
\end{equation}
where  \(\alpha+1\) is the multi-index with components
\(\alpha_{i}+1\). Thus  
\begin{equation}\nn
\widehat{P}(\rho) = Q^{N+n}\left(\frac{1}{\rho}\right)
\end{equation}
where $Q = Q^{N+n}$ is the  homogeneous polynomial of
degree $N+n$ with coefficients  \(q_{\alpha+1}= \alpha!p_{\alpha}\) . 
The main assertion of Theorem \ref{LTharmonic} is that
\(\hat{P}(\rho)\) does not vanish on any open subset of the
variety  \(\rho\cdot\rho = 0\). This is equivalent to the
assertion that  the polynomial \(Q(\eta)\) does not vanish
identically on any open subset of 
\begin{eqnarray}
  \nn
 \left\{\frac{1}{\eta} \cdot \frac{1}{\eta} = 0\right\}
\end{eqnarray}
where
\begin{equation}
\label{rhoDrhoToEta}
 \rho\cdot\rho = \frac{1}{\eta} \cdot \frac{1}{\eta} = \frac{\sigma_{n-1}(\eta^2)}{\sigma_n(\eta^2)} = \frac{\sigma_{n-1}(\eta^2)}{\sigma_n^2(\eta)}.
\end{equation}

If $P$ is harmonic, $Q(\eta) = \widehat{P}(\frac{1}{\eta})$ has an additional representation.

\begin{lemma}\label{sec:lemma-ibp-form}
If $P$ is harmonic and homogeneous, then
\begin{equation}
\label{ihatrep}
Q(\eta) = \widehat{P}(\frac{1}{\eta}) = \frac{\sigma_n(\eta)}{\sigma_{n-1}(\eta^2)} \sum_{i=1}^n \left( P_{\hat{i}} + \eta_i Q_{\hat{i}} \right)
\end{equation}
where $P_{\hat{i}}$ and $Q_{\hat{i}}$ are homogeneous polynomials of degree $N+2n-2$, $N+2n-3$, respectively, which do not depend on $\eta_i$.
\end{lemma}
\begin{proof}
  We will prove \eqref{ihatrep} on the open set where
  $\Re \rho > 0$ and that  $\rho
  \cdot \rho \neq 0$. Because  \(Q\) is a polynomial in  \(\eta\) ,
  the right hand side of \eqref{ihatrep} is also a
  polynomial, and the identity must hold everywhere.\\

  We start with (\ref{eq:26}),
  integrate by parts, and recall that $P$ is harmonic,
\begin{eqnarray}\nn
\widehat{P}(\rho) &=& \int_{x>0} \frac{\Delta e^{-\rho\cdot
    x}}{\rho \cdot \rho} P(x) dx
\\\nn
&=& \frac{1}{\rho \cdot \rho} \sum_{i=1}^n \left(
  \Int_{\substack{x_i=0\\x_{\hat{i}}>0}} e^{-\rho\cdot x}
  \left( \rho_i P + \frac{\d}{\d x_i} P \right) +
  \Int_{x>0} e^{-\rho\cdot x} \Delta P \right)
\\\nn
&=& \frac{1}{\rho \cdot \rho} \sum_{i=1}^n \left( \rho_i
  \int_{x_{\hat{i}}>0} e^{-\rho_{\hat{i}} \cdot
    x_{\hat{i}}} P\Big|_{ x_i=0} + \Int_{x_{\hat{i}}>0}
  e^{-\rho_{\hat{i}} \cdot x_{\hat{i}}} \frac{\d}{\d x_i}
  P\Big|_{ x_i=0} \right)
\\\nn
&=& \frac{1}{\rho \cdot \rho} \sum_{i=1}^n \left( \rho_i
  \tilde{P_i}\left(\frac{1}{\rho_{\hat{i}}}\right)
 + \tilde{Q_i}\left(\frac{1}{\rho_{\hat{i}}}\right) \right)
\end{eqnarray}
where  \( \widetilde{P_i}\) and  \( \widetilde{Q_i}\)
simply denote polynomials in  \(n-1\) variables. 

Recalling \eqref{rhoDrhoToEta}, the polynomial  \(Q\) then satisfies  
\begin{eqnarray}\nn
  Q(\eta) &=&
  \frac{(\sigma_n(\eta))^2}{\sigma_{n-1}(\eta^2)} 
  \sum_{i=1}^n \left( \frac{1}{\eta_i}
    \widetilde{P_i}(\eta_{\hat{i}}) 
         +  \widetilde{Q_i}(\eta_{\hat{i}}) \right) \\\nn
  &=& \frac{\sigma_n(\eta)}{\sigma_{n-1}(\eta^2)} 
  \sum_{i=1}^n \left( \frac{\sigma_n(\eta)}{\eta_i}
    \widetilde{P_i}(\eta_{\hat{i}})
    + \sigma_n(\eta) \widetilde{Q_i}(\eta_{\hat{i}}) \right) \\\nn
  &=& \frac{\sigma_n(\eta)}{\sigma_{n-1}(\eta^2)}
  \sum_{i=1}^n \left( P_i(\eta_{\hat{i}}) +
    \eta_i Q_i(\eta_{\hat{i}}) \right),
\end{eqnarray}
so that  \( P_i(\eta_{\hat{i}}) =
\sigma_{n-1}(\eta_{\hat{i}})
\widetilde{P_i}(\eta_{\hat{i}})\) and
\(Q_i(\eta_{\hat{i}})=\sigma_{n-1}(\eta_{\hat{i}})\widetilde{Q_i}(\eta_{\hat{i}})\)
are the polynomials \(P_{\hat{i}}\) and \(Q_{\hat{i}}\) of
(\ref{ihatrep}), and the proof is finished.
\end{proof}

The irreducibility of the denominator,  \(\sigma_{n-1}(\eta^2)\)  in
(\ref{ihatrep}) will play a role in several parts of our
proof, so we prove this fact here:

\begin{lemma}
\label{Rirreducible}
If $n \geq 3$,
$\sigma_{n-1}(\eta^2)$ is an irreducible polynomial. If
$n=2$, then $\sigma_{n-1}(\eta^2) = \eta_1^2 + \eta_2^2 =
(\eta_1 - i \eta_2)(\eta_1 + i \eta_2)$.
\end{lemma}
\begin{proof}
  The statement for \(n=2\) is obvious.  We will prove
  this lemma for \(n\ge 3\) by induction, making use of
  the identity
\begin{equation}\nn
\sigma_{n-1}(\eta^2) = \eta_1^2 \sigma_{n-2}(\eta_{\hat{i}}^2) + \sigma_{n-1}(\eta_{\hat{i}}^2)
\end{equation}
for $\eta\in\C^n$. If $\sigma_{n-1}(\eta^2)$ factors,
and one factor does not depend on $\eta_1$, then 
we must have
\begin{equation}\nn
\eta_1^2 \sigma_{n-2}(\eta_{\hat1}^2) + \sigma_{n-1}(\eta_{\hat1}^2)=\sigma_{n-1}(\eta^2) = p_{\hat1} ( \eta_1^2 q_{\hat1} + r_{\hat1})
\end{equation}
where $p_{\hat1}, q_{\hat1},$ and $r_{\hat1}$ are
non-constant polynomials independent of $\eta_1$. Equating
coefficients of $\eta_1^2$ gives
\begin{equation}\nn
\sigma_{n-2}(\eta_{\hat1}^2) = p_{\hat1}q_{\hat1},
\end{equation}
which contradicts the induction hypothesis because
$\sigma_{n-2}(\eta_{\hat1}^2) = \sigma_{m-1}(\xi^2)$ with
$m=n-1$ and $\xi = \eta_{\hat1} \in \C^m$.\\

On the other hand, if both factors depend on $\eta_1$, i.e.
\begin{equation}\nn
\eta_1^2 \sigma_{n-2}(\eta_{\hat1}^2) + \sigma_{n-1}(\eta_{\hat1}^2) = (\eta_1 p_{\hat1} + r_{\hat1})(\eta_1q_{\hat1} + s_{\hat1}).
\end{equation}
Equating coefficients of $\eta_1^2$ again gives
\begin{equation}\nn
\sigma_{n-2}(\eta_{\hat1}^2) = p_{\hat1} q_{\hat1}
\end{equation}
and contradicts the induction hypothesis.\\

We finish the induction by verifying irreduciblity in the
case $n=3$. In this case,
\begin{equation}\nn
\sigma_2(\eta^2) =
\eta_1^2(\eta_2 + i \eta_3)(\eta_2 - i \eta_3) + \eta_2^2\eta_3^2.
\end{equation}
If 
\begin{eqnarray}
  \nn
  p_{\hat1}(\eta_1^2 q_{\hat1} + r_{\hat1}) =
\sigma_2(\eta^2) = 
\eta_1^2(\eta_2 + i \eta_3)(\eta_2 - i \eta_3)
+ \eta_2^2\eta_3^2
\end{eqnarray}
equating the coefficients of $\eta_1^2$, we see again that
$p_{\hat1}$ must divide
$(\eta_2+i\eta_3)(\eta_2-i\eta_3)$. Equating the
coefficients of the terms that do not involve
\(\eta_{1}^{2}\), tells us that $p_{\hat1}$ also divides
$\eta_2^2\eta_3^2$, but this is impossible because the two
have prime factorizations without common factors.\\

If, on the other hand,  \((\eta_1 p_{\hat1} + r_{\hat1})(\eta_1 q_{\hat1} +s_{\hat1})
=\sigma_2(\eta^2)\), expanding both sides of the equation shows that
\begin{equation}\nn
\eta_1^{2} p_{\hat1}q_{\hat1} + \eta_1 (q_{\hat1}r_{\hat1}
+p_{\hat1}s_{\hat1}) + 
= \eta_1^2(\eta_2 + i \eta_3)(\eta_2 - i \eta_3)
+ \eta_2^2\eta_3^2+r_{\hat1}s_{\hat1}
\end{equation}
then
\begin{equation}\nn
p_{\hat1}q_{\hat1} = (\eta_2 + i \eta_3)(\eta_2 - i \eta_3)
\end{equation}
which implies that $p_{\hat1}$ must be a constant multiple
of either $(\eta_2+i\eta_3)$ or $(\eta_2 - i \eta_3)$ and
$q_{\hat1}$ must be a constant multiple of the other. Also
\begin{equation}\nn
q_{\hat1}r_{\hat1} = - p_{\hat1}s_{\hat1}
\end{equation}
so that $p_{\hat1}$ must divide $r_{\hat1}$ because it doesn't divide $q_{\hat1}$. However
\begin{equation}\nn
r_{\hat1} s_{\hat1} = \eta_2^2 \eta_3^2
\end{equation}
does not have $(\eta_2 \pm i\eta_3)$ as a factor, so this
is also impossible and the proof is complete.
\end{proof}

We will need two more propositions for the proof of 
Theorem \ref{LTharmonic}. The first follows easily from
the previous lemma.
\begin{proposition}
\label{proposition1}
If $P$ is harmonic and homogeneous, and 
$\widehat{P}(\rho)$ vanishes identically 
on $\{\rho \cdot \rho = 0\}$, then $\sigma_{n-1}(\eta^2)$
divides the polynomial $Q (\eta) =
\widehat{P}(\frac{1}{\eta})$.
\end{proposition}
\begin{proof}
  Because \(\hat{P}\) vanishes on \(\{\rho \cdot \rho =
  0\}\), it folllows from (\ref{rhoDrhoToEta}) that \(Q\)
  vanishes on the set \(\{\sigma_{n-1}(\eta^2) =
  0\}\setminus\{\sigma_n(\eta^{2}) = 0\}\). Therefore, the
  product \(\sigma_n(\eta^{2})Q(\eta^{2})\) vanishes on the
  entire variety \(\{ \sigma_{n-1}(\eta^2) = 0\}\), and
  hence must be divisible by
  \(\sigma_{n-1}(\eta^2)\) by Hilbert's Nullstellensatz. For  \(n\ge 3\),  \(\sigma_{n-1}(\eta^2)\)
  is irreducible and doesn't divide \(\sigma_n(\eta^2)\),
  so it must divide \(Q\). In the case  \(n=2\),
  \(\sigma_{n-1}(\eta^2)\) has two factors; neither factor
  divides \(\sigma_n(\eta^2)\), so both divide  \(Q\). 
\end{proof}

The proof of the next proposition will not be so easy,
\begin{proposition}
\label{proposition2}
$\sigma_{n-1}(\eta^2)$ cannot divide any polynomial $Q$ of the form \eqref{ihatrep}.
\end{proposition}
\noindent
but the proof of Theorem \ref{LTharmonic} is an immediate consequence.

\begin{proof}[Proof of Theorem \ref{LTharmonic}]
  If \(n\ge3\), the hypothesis of Theorem \ref{LTharmonic}
  is that \(\hat{P}\) vanishes on an open subset of
  $\{\rho \cdot \rho = 0\}$, which means that \(Q\),
  vanishes on an open subset of the irreducible variety
  \(\sigma_{n-1}(\eta^2)=0\). But this means that  \(Q\)
  vanishes on the whole variety by \cite[p.91]{Ritt} or \cite{Waerden} and that
  $\sigma_{n-1}(\eta^2)$ divides  \(Q\), contradicting
  Proposition \ref{proposition2}.

  If  \(n=2\), we have the same hypothesis for each of
  the irreducible factors,
  \(\rho_{1}-i\rho_{2}\) and \(\rho_{1}+i\rho_{2}\), so we
  may conclude that each divides  \(Q\), and therefore
  that  \(Q\) is divisible by their product \(\sigma_{n-1}(\eta^2)\).  
\end{proof}

\begin{proof}[Proof of Proposition \ref{proposition2}]

We will make essential use of the the fact that
\(\sigma_{n-1}(\eta^2)\) is even in each component
\(\eta_{j}\)  of  \(\eta\).

\begin{lemma}\label{sec:expansion-lemma}
  Every polynomial \(R(\eta)\)  has a unique
  decomposition into a sum 
  \begin{equation}
    \nn
    R(\eta) = \sum_{\tau\in\{0,1\}^{n}}\eta^{\tau}R_{\tau}(\eta^{2})
  \end{equation}
where  \(\tau\) is a multi-index with each component equal
to  \(0\) or  \(1\). If  \(R\) has the special form
\(R=\sum_{i}(\eta_{i}P_{\hat{i}}+Q_{\hat{i}})\), then each
of the  coefficients \(R_{\tau}(\eta^{2})\) has the
special form
\begin{eqnarray}
  \label{eq:39}
  R_{\tau}=\sum_{i}S_{\h{i}}(\eta_{\h{i}}^{2})
\end{eqnarray}
\end{lemma}
\begin{proof} We express  \(R\) as a sum of monomials, 
  \begin{eqnarray}
    \nn
    R(\eta) &=& \sum_{\alpha}p_{\alpha}\eta^{\alpha}
    \\\na{group the terms that are even or odd for each
      \(\eta_{i}\) together}\nn
    &=&
    \sum_{\tau\in\{0,1\}^{n}}\left(\sum_{\alpha\equiv_{2}\tau}
      (p_{\alpha}\eta^{\alpha})\right)
    \\\na{and remove a single power of  \(\eta_{i}\) from
      each monomial that is odd in  \(\eta_{i}\)}\label{eq:9}
    &=&
    \sum_{\tau\in\{0,1\}^{n}}\left(\sum_{\alpha\equiv_{2}\tau}
      (p_{\alpha}\eta^{\alpha-\tau})\right)\eta^{\tau}
    \\\na{so that the summands in the parentheses
      contains only even powers\espace}\nn
    &=&
    \sum_{\tau\in\{0,1\}^{n}}R_{\tau}(\eta^{2})\eta^{\tau}
  \end{eqnarray}
  The explicit formula for  each \(R_{\tau}\) in (\ref{eq:9}) implies that the
  decomposition is unique. Suppose now that \(R\) has the
  special form
  \(\sum_{i}(\eta_{i}P_{\hat{i}}+Q_{\hat{i}})\), we can
  first decompose each of the \(Q_{\hat{i}}\) and the
  \(P_{\hat{i}}\). 
  \begin{eqnarray}\nn
    \eta_{i}P_{\h{i}}+Q_{\hat{i}}
    &=&
    \eta_{i}\left(\sum_{\tau_{\h{i}}\in\{0,1\}^{n-1}}
      P_{\tau_{\h{i}}}(\eta_{\h{i}}^{2})\eta^{\tau_{\h{i}}}\right)
  +
  \sum_{\tau_{\h{i}}\in\{0,1\}^{n-1}}
  Q_{\tau_{\h{i}}}(\eta_{\h{i}}^{2})\eta^{\tau_{\h{i}}}
  \\\nn
  &=&
    \sum_{\tau_{\h{i}}\in\{0,1\}^{n-1}}
    P_{\tau_{\h{i}}}(\eta_{\h{i}}^{2})\left(\eta^{\tau_{\h{i}}}\eta_{i}^{1}\right)
  +\sum_{\tau_{\h{i}}\in\{0,1\}^{n-1}}
  Q_{\tau_{\h{i}}}(\eta_{\h{i}}^{2})\left(\eta^{\tau_{\h{i}}}\eta_{i}^{0}\right)
      \end{eqnarray}
which shows that each summand
\(\eta_{i}P_{\hat{i}}+Q_{\hat{i}}\) has a decompostion
where the coefficients  of \(\eta^{\tau}\) are indepedent
of  \(\eta_{i}\). Thus the sum has coefficients which are
sums of such functions. 
\end{proof}

\begin{lemma}\label{sec:lemma-1}
  If a polynomial  \(S(\eta^{2})\) divides
  \(\displaystyle R(\eta)=
  \sum_{\tau\in\{0,1\}^{n}}\eta^{\tau}R_{\tau}(\eta^{2})\),
  then   \(S\) divides each  \(R_{\tau}\).  
\end{lemma}
\begin{proof}
  Suppose that 
  \begin{eqnarray}\nn
    R(\eta) &=& S(\eta^{2})C(\eta)
\\\na{expand both  \(R\) and  \(C\) as in Lemma
  \ref{sec:expansion-lemma}\espace}\nn  
\sum_{\tau\in\{0,1\}^{n}}\eta^{\tau}R_{\tau}(\eta^{2})
&=&
S(\eta^{2})\sum_{\tau\in\{0,1\}^{n}}\eta^{\tau}C_{\tau}(\eta^{2})
\\\nn
&=&
\sum_{\tau\in\{0,1\}^{n}}\eta^{\tau}S(\eta^{2})C_{\tau}(\eta^{2})    
  \end{eqnarray}
and now use the uniqueness of the expansion to equate the
coefficients of each monomial  \(\eta^{\tau}\). 
\end{proof}

The last ingredient necessary for the proof of Proposition
\ref{proposition2} is
\begin{proposition}\label{sec:prop3}
   \(\sigma_{n-1}^{2}(s)\) does not divide any polynomial
   of the form    \(T(s)=\sum T_{\hat{i}}\)  unless  \(T\)
   is identically zero. 
\end{proposition}

Before giving its proof, we use it to finish the proof of
Proposition \ref{proposition2}.  If, as in
the hypothesis of Proposition \ref{proposition2},
\(\sigma_{n-1}^{2}(\eta^{2})\) divides  \(R = \sum
(\eta_{i}P_{\hat{i}}+Q_{\hat{i}})\), then, according to
Lemma \ref{sec:lemma-1},
\(\sigma_{n-1}^{2}(\eta^{2})\) divides each of the
\(R_{\tau}\) in the expansion of Lemma
\ref{sec:expansion-lemma}, and each  \(R_{\tau}\) 
has the special form  (\ref{eq:39}). Proposition
\ref{sec:prop3} says that  this is impossible (the
variable  \(s\) replaces \(\eta^{2}\))   and thus
finishes the proof of Proposition \ref{proposition2}.
\end{proof}

\begin{proof}[Proof of Proposition \ref{sec:prop3}]
 We will prove the proposition by induction on the number
 of independent variables. We will  expand all polynomials
 as polynomials in the single variable  \(s_{1}\) with
 coefficients that depend on the other variables. We begin
 with 

 \begin{eqnarray}
   \nn
   \sigma_{n-1}(s) &=& s_{1}\sigma_{n-2}(s_{\h1}) +
   \sigma_{n-1}(s_{\h1})
\\\nn
 \sigma_{n-1}^{2}(s)&=& s_{1}^{2}\sigma_{n-2}^{2}(s_{\h1})
 + s_{1}2\sigma_{n-2}(s_{\h1})\sigma_{n-1}(s_{\h1}) + \sigma_{n-1}^{2}(s_{\h1})
 \end{eqnarray}
 
If a general polynomial  \(T(s)\) has
\(\sigma_{n-1}^{2}(s)\) as a factor, then expanding the
equality  \(T=\sigma_{n-1}^{2}C\) in powers of
\(s_{1}\) gives 
\begin{eqnarray}
  \nn
\sum_{k=0}^{N}s_{1}^{k}T^{k}(s_{\h1}) = \Big( s_{1}^{2}\sigma_{n-2}^{2}(s_{\h1})
 + s_{1}2\sigma_{n-2}(s_{\h1})\sigma_{n-1}(s_{\h1}) + \sigma_{n-1}^{2}(s_{\h1})\Big)
                                   \left(\sum_{k=0}^{N-2}s_{1}^{k}C^{k}(s_{\h1})\right)
\end{eqnarray}
Equating coefficients of  powers of \(s_{1}\) gives
\begin{eqnarray}\label{eq:34}
  T^{N}(s_{\h1}) &=&
  \sigma_{n-2}^{2}(s_{\h1})C^{N-2}(s_{\h1})
\\\na{and, for  \(j=1\ldots(N-2),\)\espace}\label{eq:41}
  T^{N-j}(s_{\h1}) &=&
  \sigma_{n-2}^{2}(s_{\h1})C^{N-2-j}(s_{\h1}) + \ldots
  \end{eqnarray}
where the  \(\ldots\) indicate terms involving
  \(C^{k}\) for  \(k>N-2-j\). We won't need to use the
  equations for  \(T^{1}\) and  \(T^{0}\).\\

Now, if  \(T\) has the special form  \(T=\sum T_{\h{i}}\),
with the  \(T_{\h{i}}\) independent of  \(s_{i}\), then  
each of the  \(T^{k}\), except  \(T^{0}\),  will have the special form
\begin{eqnarray}
  \nn
  T^{k} = \sum_{i=2}^{n} T^{k}_{\h{1},\h{i}}
\end{eqnarray}
where the subscripts indicate that  \(T^{k}_{\h{1},\h{i}}\) is independent of both
\(\eta_{1}\) and  \(\eta_{i}\). Thus equation
(\ref{eq:34}) becomes
\begin{eqnarray}
  \label{eq:36}
  \sum_{i=2}^{n} T^{N}_{\h{1},\h{i}} =  \sigma_{n-2}^{2}(s_{\h1})C^{N-2}(s_{\h1})
\end{eqnarray}
but this is exactly the hypothesis of the proposition for
one fewer dimension. If we let  \(\beta=s_{\h1}\) and
\(m=n-1\), then (\ref{eq:36}) becomes
\begin{eqnarray}
  \nn
   \sum_{i=1}^{m} T^{N}_{\h{i}}(\beta) = \sigma_{m-1}^{2}(\beta)C^{N-2}(\beta)
\end{eqnarray}
and the induction hypothesis guarantees that \(C^{N-2}\)
and \(T^{N}\) are both identically zero. Once we know that
\(C^{N-2}\) is zero, we may conclude that the term
represented by the \(\ldots\) in equation (\ref{eq:41})
for \(T^{N-1}\) is zero, and repeat the argument to
conclude that \(C^{N-3}\) and \(T^{N-2}\) are zero. We
continue in this manner to conclude that all the
\(C^{k}\), and
therefore all the  \(T^{k}\) are zero. \\

Finally, we verify the proposition in the case
\(n=2\). In this case, we must check that the equality
below
\begin{eqnarray}
 \nn
  p_{N}x^{N} + q_{N}y^{N} = (x+y)^{2}\sum_{k=0}^{N-2} c_{k}x^{k}y^{N-2-k}
\end{eqnarray}
is only possible if  \(p_{N}\), \(q_{N}\), and all the  \(c_{k}\)  are
zero. Equating powers of  \(x\) and  \(y\) give
\begin{eqnarray}
\nn
  p_{N} &=& c_{N-2}
\\\nn
0  &=& 2c_{N-2}+ c_{N-1}
\\\na{for  \(j = 2\ldots (N-2)\) }\nn
0  &=& c_{N-(j+2)} +2c_{N-(j+1)}+ c_{N-j}
\\\na{and}\nn
0  &=& c_{1} +2c_{0}
\\\nn
q_{N} &=& c_{0}
\end{eqnarray}
Discarding the first and last equations gives the
invertible tridiagonal system
\begin{eqnarray}\nn
  \begin{pmatrix}
    0\\\vdots\\\vdots\\0
  \end{pmatrix}
=
\begin{pmatrix}
  2&1&0&\ldots&0
\\
1&2&1&\ldots&0
\\
0&1&2&\ldots&0
\\
\vdots&&&&\vdots
\\
0&\ldots&0&1&2
\end{pmatrix}
\begin{pmatrix}
c_{N-2}\\c_{N-1}\\\vdots\\c_{1}\\c_{0}
\end{pmatrix}
  \end{eqnarray}
whence we conclude that all the  \(c_{k}\) are zero. This
finishes the proof of the proposition. 
\end{proof}

\bibliographystyle{plain}
\bibliography{LpHatBibl}

\begin{thebibliography}{10}

\bibitem{AgmonHormander}
Shmuel Agmon and Lars H{\"o}rmander.
\newblock Asymptotic properties of solutions of differential equations with
  simple characteristics.
\newblock {\em J. Analyse Math.}, 30:1--38, 1976.

\bibitem{BL}
J{\"o}ran Bergh and J{\"o}rgen L{\"o}fstr{\"o}m.
\newblock {\em Interpolation spaces. {A}n introduction}.
\newblock Springer-Verlag, Berlin, 1976.
\newblock Grundlehren der Mathematischen Wissenschaften, No. 223.

\bibitem{CakoniGintidesHaddar}
Fioralba Cakoni, Drossos Gintides, and Houssem Haddar.
\newblock The existence of an infinite discrete set of transmission
  eigenvalues.
\newblock {\em SIAM J. Math. Anal.}, 42(1):237--255, 2010.

\bibitem{coltonKirsch96}
David Colton and Andreas Kirsch.
\newblock A simple method for solving inverse scattering problems in the
  resonance region.
\newblock {\em Inverse Problems}, 12(4):383--393, 1996.

\bibitem{RellichLemma}
David Colton and Rainer Kress.
\newblock {\em Inverse acoustic and electromagnetic scattering theory},
  volume~93 of {\em Applied Mathematical Sciences}.
\newblock Springer-Verlag, Berlin, 1992.

\bibitem{coltonMonk88}
David Colton and Peter Monk.
\newblock The inverse scattering problem for time-harmonic acoustic waves in an
  inhomogeneous medium.
\newblock {\em Quart. J. Mech. Appl. Math}, 41(1):97--125, 1988.

\bibitem{coltonPaivarintaSylvester07}
David Colton, Lassi P{\"a}iv{\"a}rinta, and John Sylvester.
\newblock The interior transmission problem.
\newblock {\em Inverse Probl. Imaging}, 1(1):13--28, 2007.

\bibitem{HKOP1}
Michael Hitrik, Katsiaryna Krupchyk, Petri Ola, and Lassi P{\"a}iv{\"a}rinta.
\newblock Transmission eigenvalues for operators with constant coefficients.
\newblock {\em SIAM J. Math. Anal.}, 42(6):2965--2986, 2010.

\bibitem{KellerLewis}
Joseph~B. Keller and Robert~M. Lewis.
\newblock Asymptotic methods for partial differential equations: the reduced
  wave equation and {M}axwell's equations.
\newblock In {\em Surveys in applied mathematics, {V}ol.\ 1}, volume~1 of {\em
  Surveys Appl. Math.}, pages 1--82. Plenum, New York, 1995.

\bibitem{kirschGrinberg08}
Andreas Kirsch and Natalia Grinberg.
\newblock {\em The factorization method for inverse problems}, volume~36 of
  {\em Oxford Lecture Series in Mathematics and its Applications}.
\newblock Oxford University Press, Oxford, 2008.

\bibitem{paivarintaSylvester08}
Lassi P{\"a}iv{\"a}rinta and John Sylvester.
\newblock Transmission eigenvalues.
\newblock {\em SIAM J. Math. Anal.}, 40(2):738--753, 2008.

\bibitem{Ritt}
Joseph~Fels Ritt.
\newblock {\em Differential Equations from the Algebraic Standpoint}.
\newblock American mathematical society. Colloquium publications, 1932.

\bibitem{RuizNotes}
Alberto Ruiz.
\newblock {Harmonic Analysis and Inverse Problems}.
\newblock {lecture notes,
  http://www.uam.es/gruposinv/inversos/publicaciones/Inverseproblems.pdf},
  accessed 2012.

\bibitem{SommerfeldOptics}
Arnold Sommerfeld.
\newblock {\em Optics. {L}ectures on theoretical physics, {V}ol. {IV}}.
\newblock Academic Press Inc., New York, 1954.
\newblock Translated by O. Laporte and P. A. Moldauer,.

\bibitem{Stein}
Elias~Menachem Stein.
\newblock {\em Singular integrals and differentiability properties of
  functions}.
\newblock Princeton Mathematical Series, No. 30. Princeton University Press,
  Princeton, N.J., 1970.

\bibitem{sylvesteruhlmann87}
John Sylvester and Gunther Uhlmann.
\newblock A global uniqueness theorem for an inverse boundary value problem.
\newblock {\em Ann. of Math. (2)}, 125(1):153--169, 1987.

\bibitem{Triebel1}
Hans Triebel.
\newblock {\em Theory of function spaces}, volume~78 of {\em Monographs in
  Mathematics}.
\newblock Birkh\"auser Verlag, Basel, 1983.

\bibitem{Waerden}
Bartel~Leendert van~der Waerden.
\newblock Zur algebraischen {G}eometrie. {III}.
\newblock {\em Math. Ann.}, 108(1):694--698, 1933.

\end{thebibliography}

\end{document}